\documentclass[11pt]{article}
\usepackage[a4paper,margin=2.7cm]{geometry}
\usepackage[comma,square,sort&compress,numbers]{natbib}
\usepackage{amsmath, amsthm, amssymb}
\usepackage{graphicx}
\usepackage{caption}
\usepackage{subcaption}
\usepackage{pgf,tikz}
\usepackage{txfonts}
\usepackage[pdfpagemode=UseOutlines]{hyperref}
\hypersetup{
    colorlinks=false,
    pdfborder={0 0 0},
    pdfauthor={B. de Rijk, G. Schneider},
   pdftitle={Global existence and decay in multi-component reaction-diffusion-advection systems with different velocities: oscillations in time and frequency}
}
\usepackage{enumitem}

\makeatletter
\def\namedlabel#1#2{\begingroup
    #2%
    \def\@currentlabel{#2}%
    \phantomsection\label{#1}\endgroup
}
\makeatother
\newcommand{\R}{\mathbb{R}}\newcommand{\Nn}{\mathbb{N}}\newcommand{\F}{\mathcal{F}}\newcommand{\C}{\mathbb{C}}

\newcommand{\D}{\mathcal{D}}

\newcommand{\J}{\mathcal{J}}

\newcommand{\N}{\mathcal{N}}

\newcommand{\Ce}{\mathcal{C}}

\newcommand{\I}{\mathcal{I}}

\newcommand{\ri}{\mathrm{i}}

\newcommand{\re}{\mathrm{e}}
\numberwithin{equation}{section}

\def\de{\mathrm{d}}

\let\epsilon\varepsilon

\let\phi\varphi
\def\xt{\zeta}
\def\vt{z} 

\newtheorem{theorem}{Theorem}[section]

\newtheorem{proposition}[theorem]{Proposition}

\newtheorem{remark}[theorem]{Remark}

\makeatletter
\renewenvironment{proof}[1][\proofname] {\par\pushQED{\qed}\normalfont\topsep6\p@\@plus6\p@\relax\trivlist\item[\hskip\labelsep\bfseries#1\@addpunct{.}]\ignorespaces}{\popQED\endtrivlist\@endpefalse} \makeatother

\title{Global existence and decay in multi-component reaction-diffusion-advection systems with different velocities: oscillations in time and frequency}

\author{Bj\"orn de Rijk\footnote{\textsc{Zentrum Mathematik, Technische Universit\"at M\"unchen, Boltzmannstr.~3, 85748 Garching bei M\"unchen, Germany. ~}
  \textit{E-mail address:} \texttt{bjoern.de-rijk@tum.de}} ~and Guido Schneider\footnote{\textsc{Institut f\"ur Analysis, Dynamik und Modellierung, Universit\"at Stuttgart, Pfaffenwaldring~57, 70569 Stuttgart, Germany. ~} \textit{E-mail address:} \texttt{guido.schneider@mathematik.uni-stuttgart.de}}}

\begin{document}

\maketitle

\begin{abstract}
It is well-known that quadratic or cubic nonlinearities in reaction-diffusion-advection systems can lead to growth of solutions with small, localized initial data and even finite time blow-up. It was recently proved, however, that, if the components of two nonlinearly coupled reaction-diffusion-advection equations propagate with different velocities, then quadratic or cubic mixed-terms, i.e.~nonlinear terms with nontrivial contributions from both components, do not affect global existence and Gaussian decay of small, localized initial data. The proof relied on pointwise estimates to capture the difference in velocities. In this paper we present an alternative method, which is better applicable to multiple components. Our method involves a nonlinear iteration scheme that employs $L^1$-$L^p$ estimates in Fourier space and exploits oscillations in time and frequency, which arise due to differences in transport. Under the assumption that each component exhibits different velocities, we establish global existence and decay for small, algebraically localized initial data in multi-component reaction-diffusion-advection systems allowing for cubic mixed-terms and nonlinear terms of Burgers' type.

\end{abstract}

\section{Introduction}

Let $n \in \Nn_{\geq 2}$. We consider reaction-diffusion-advection (RDA) systems on the real line of the form
\begin{align}
\partial_t u &= \D \partial_{xx} u + \Ce \partial_x u + f\left(u,\partial_x u\right), \qquad u(x,t) \in \R^n, t \geq 0, x \in \R, \label{GRD}
\end{align}
where $f \colon \R^n \times \R^n \to \R^n$ is a smooth nonlinearity satisfying $f(0,0), Df(0,0) = 0$ and
\begin{align*}\D := \mathrm{diag}(d_1,\ldots,d_n), \qquad \Ce := \mathrm{diag}(c_1,\ldots,c_n), \end{align*}
are diagonal matrices with diffusion coefficients $d_i > 0$ and velocities $c_i \in \R$. System~\eqref{GRD} describes the dynamics of $n$ diffusive species, which are each subject to a spatial species-dependent drift, such that the interactions within and among species are purely nonlinear. Such nonlinear interactions, as well as differences in diffusion rates and spatial drifts, arise naturally in various applications.

For instance, in transport-reaction problems in porous media~\cite{KRAU} one distinguishes between mobile species undergoing advection-diffusion and immobile species. Here, mass-action kinetics leads to purely nonlinear interactions~\cite{ERDI,GROE}. In addition, reaction-diffusion models describing the flow down an inclined plane~\cite{CAD,HSZ} exhibit differences in velocities as perturbations of the underlying background state are advected at a different speed than bifurcating periodic patterns. Here, at the onset of a hydrodynamic instability, the reaction terms are purely nonlinear. Finally, systems of the form~\eqref{GRD} also arise in mathematical applications. For instance, they capture the critical dynamics at the Eckhaus boundary, where periodic wave-train solutions to reaction-diffusion systems destabilize through a Hopf instability. The Eckhaus boundary plays an important role in the theory of pattern formation~\cite{ahlers}. We refer to our prior paper~\cite{RSDV} for further discussion and literature references.

We are interested in the effect of the velocities and the nonlinearity in~\eqref{GRD} on the long-time dynamics of small, localized initial data. It is beneficial to label nonlinear terms depending on their (possible) effect. To each term of the form
\begin{align} \prod_{i = 1}^n \prod_{j = 1}^n u_i^{a_i} \left(\partial_x u_j\right)^{b_j}, \qquad a_i, b_j \in \Nn_{\geq 0}, \label{Gnon1}\end{align}
we assign a number
\begin{align*} p := \sum_{i = 1}^n a_i + \sum_{j = 1}^n 2b_j.\end{align*}
We call the nonlinear term~\eqref{Gnon1} \emph{relevant} if $p < 3$, \emph{irrelevant} if $p > 3$ and \emph{marginal} if $p = 3$. For instance, any \emph{Burgers'-type term}, i.e.~any term of the form $\partial_x(u_i^2)$ with $i \in \{1,\ldots,n\}$, is marginal. Thus, any smooth nonlinearity can be labeled relevant, marginal or irrelevant by looking at the leading-order term of its power series expansion. Relevant and marginal terms in system~\eqref{GRD} can lead to growth and even finite-time blow-up of solutions with small, localized initial data; see~\cite[Section 1.1]{RSDV} for references. On the other hand, if the nonlinearity in~\eqref{GRD} is irrelevant, then solutions to~\eqref{GRD} with small, localized initial conditions always exist globally and exhibit diffusive Gaussian-like decay as one would expect from the linear dynamics only~\cite{PONC,ZHENG}. This classification of smooth nonlinearities was introduced in~\cite{BKL} and can be extended to $d$ spatial dimensions by replacing the critical threshold $p = 3$ by $p = 1 + \frac{2}{d}$, see also~\cite[Section 2]{UEC9}. The critical threshold $p = 1 + \frac{2}{d}$ is known as the \emph{Fujita exponent} after its first occurrence in the paper~\cite{FUJI} of Fujita.

We showed in~\cite{RSDV} that, if~\eqref{GRD} has two components that propagate with different velocities, then there is, besides the irrelevant ones, an additional class of nonlinear terms which are harmless. These so-called \emph{mixed-terms} have contributions from different components, i.e.~they are of the form~\eqref{Gnon1} such that there exist $i,j \in \{1,\ldots,n\}$ with $i \neq j$ and $(a_i + b_i)(a_j + b_j) \neq 0$. It was proved in~\cite{RSDV} that, if $n = 2$, all relevant and marginal terms in the nonlinearity are of mixed type and it holds $c_1 \neq c_2$, then solutions to~\eqref{GRD} with exponentially or algebraically localized initial data exist globally and decay in time with rate $t^{-1/2}$. To capture the effect of spatial transport between components, the proof in~\cite{RSDV} exploits \emph{pointwise estimates}~\cite{ZUH,HOWZUM}.

As pointed out in~\cite[Section 6]{RSDV}, it is not straightforward to extend the pointwise analysis to multiple components. In this paper, we present an alternative method which naturally applies to multiple component RDA systems. We prove that, if each component propagates with a different velocity and the nonlinearity in~\eqref{GRD} contains only irrelevant terms and marginal terms, which are either of mixed or of Burgers' type, then small, algebraically localized initial data exist globally and decay in time with rate $t^{-1/2}$. In particular, our result applies to systems of viscous conservation laws, which arise frequently in continuum mechanics~\cite{COUR,LAND}, by taking the nonlinearity in~\eqref{GRD} in \emph{divergence form}, i.e.~by taking $f(u,\partial_x u) = \partial_x \left(g(u)\right)$ where $g \colon \R^n \to \R^n$ is smooth with $g(0), Dg(0) = 0$. Thereby, our result confirms the findings in~\cite{SCHONB} for the case components propagate with different velocities, and goes beyond by including marginal nonlinearities of mixed type.

Before stating our main result, we introduce the necessary functional-analytic concepts. As usual, the (nonunitary) Fourier transform $\F \colon L^2(\R,\C^n) \to L^2(\R,\C^n)$ and its inverse $\F^{-1} \colon L^2(\R,\C^n) \to L^2(\R,\C^n)$ are determined by their action on the dense subspace $L^1(\R,\C^n) \cap L^2(\R,\C^n)$ of $L^2(\R,\C^n)$, which is given by
\begin{align}
\F(u)(k) = \int_\R \re^{-\ri k x} u(x) \de x, \qquad \F^{-1}(v)(k) = \frac{1}{2\pi} \int_\R \re^{\ri k x} v(k) \de k. \label{defFourier}
\end{align}
Throughout the manuscript we commonly use $u$ to denote elements in physical space, whereas $v$ is used to denote elements in Fourier space. We consider algebraically localized initial data of the form $\F^{-1}(v_0)$, where $v_0$ lies in the weighted Sobolev space
\begin{align*} W^{1,1}_\beta(\R,\C^n) := \left\{v \in W^{1,1}(\R,\C^n) : \|v\|_{\smash{W^{1,1}_\beta}} < \infty\right\}, \end{align*}
for some $\beta \in [0,\infty)$, which we endow with the norm
\begin{align*} \|v\|_{W^{1,1}_\beta} = \left\|(1 + |\cdot|)^\beta v\right\|_1 + \left\|(1 + |\cdot|)^\beta \partial_k v\right\|_1,\end{align*}
where $\|\cdot\|_1$ denotes the $L^1$-norm. Loosely speaking, initial data of the form $\F^{-1}(v_0)$ with $v_0 \in W^{1,1}_\beta(\R,\C^n)$ can be characterized as being just more regular than differentiable and exhibiting just stronger decay than $1/(1+|x|)$ as $x \to \pm \infty$, see Remark~\ref{remspaces} for more precise statements. To ensure that the initial condition $\F^{-1}(v_0)$ is real valued, we require that $v_0 \in W^{1,1}_\beta(\R,\C^n)$ satisfies the usual reality condition $v_0(-k) = \smash{\overline{v_0(k)}}$ for all $k \in \R$.

Finally, for $j \in \Nn_{\geq 0}$, $\alpha \in (0,1]$, $\Omega \subset \R^m$ open and $X$ a Banach space, we denote by $C^j(\Omega,X)$ the vector space of $j$-times continuously differentiable functions $f \colon \Omega \to X$ and let $C^{j,\alpha}(\Omega,X) \subset C^j(\Omega,X)$ be the subspace of function in $C^j(\Omega,X)$ whose $j$-th derivatives are (uniformly) H\"older continuous with exponent $\alpha$. The subspace $C_b^j(\Omega,X) \subset C^j(\Omega,X)$ of functions in $C^j(\Omega,X)$ having bounded derivatives up to order $j$ is a Banach space when equipped with the standard norm
\begin{align*}
\|u\|_{C^j} = \max_{|\beta| \leq j} \sup_{x \in \Omega} \left\|D^\beta u(x)\right\|_X.
\end{align*}
Similarly, $C_b^{j,\alpha}(\Omega,X) := C_b^j(\Omega,X) \cap C^{j,\alpha}(\Omega,X)$ is a Banach space when endowed with the H\"older norm:
\begin{align*}
\|u\|_{C^{j,\alpha}} = \|u\|_{C^j} + \max_{|\beta| = j} \sup_{\begin{smallmatrix} x,y \in \Omega \\ x \neq y\end{smallmatrix}} \frac{\left\|D^\beta u(x) - D^\beta u(y)\right\|_X}{\|x-y\|^\alpha},\end{align*}
where $\|\cdot\|$ denotes the Euclidean norm in $\R^m$.

We are now ready to state our main result.

\begin{theorem} \label{mainresult1}
Let $\alpha > 0$. Let the coefficients in~\eqref{GRD} satisfy $d_i > 0$ and $c_i \neq c_j$ for all $i,j \in \{1,\ldots,n\}$ with $i \neq j$. Suppose that there exist constants $C \geq 1$ and $r_0 > 0$ such that the nonlinearity $f \in C^4(\R^n \times \R^n,\R^n)$ in~\eqref{GRD} satisfies
\begin{align}
\begin{split}
\|f(a,b)\| \leq C\left(\|b\|^2 + \|a\|\|b\| + \|a\|^4 + \sum_{i = 1}^n \sum_{j = 1}^n \sum_{\begin{smallmatrix} m \in \{1,\ldots,n\},\\ m \neq j\end{smallmatrix}}\! |a_i||a_j||a_m|\right),
\end{split} \label{nonlinearbounds1}
\end{align}
for all $a, b \in \R^n$ with $\|a\|,\|b\| \leq r_0$. Then, for all $\epsilon > 0$ there exists a $\delta > 0$ such that for each $v_0 \in \smash{W^{1,1}_{1+\alpha}(\R,\C^n)}$, satisfying $\|v_0\|_{\smash{W^{1,1}_1}} \leq \delta$ and $\smash{v_0(-k) = \overline{v_0(k)}}$ for all $k \in \R$,~\eqref{GRD} has a classical global solution $u \in \smash{C^{1,\frac{\alpha}{2}}}\big((0,\infty),\smash{C_b^{3,\alpha}}(\R,\R^n)\big)$ with initial condition $u(0) = \F^{-1}(v_0)$ enjoying the temporal decay estimates
\begin{align} \|u(t)\|_\infty  \leq \frac{\epsilon}{\sqrt{1 + t}}, \qquad \|\partial_x u(t)\|_\infty  \leq \frac{\epsilon}{1 + t}, \qquad \text{for } t \geq 0. \label{temporaldec}
\end{align}
in the $L^\infty$-norm. In addition, each component of $u$ is polynomially localized in an appropriate co-moving frame, i.e.~the pointwise estimate
\begin{align} \left|u_i(x,t)\right| + \frac{\sqrt{1+t}}{\ln(2+t)} \left|\partial_x u_i(x,t)\right| \leq \frac{\epsilon}{1 + |x + c_i t| + \sqrt{t}}, \label{pointdec}\end{align}
holds for all $x \in \R$, $t \geq 0$ and $i \in \{1,\ldots,n\}$.
\end{theorem}

We emphasize that the decay estimates obtained in Theorem~\ref{mainresult1} are as expected from the linear dynamics in~\eqref{GRD} only. Indeed, it is not difficult to verify that the solution $u(t)$ to the corresponding linear system
\begin{align}
\partial_t u &= \D\partial_{xx} u + \Ce \partial_x u, \qquad u(x,t) \in \R^n, t \geq 0, x \in \R, \label{GRDlin}
\end{align}
having initial condition $u(0) = \F^{-1}(v_0)$ with $v_0 \in W^{1,1}_1(\R,\C^n)$, satisfies~\eqref{temporaldec} and~\eqref{pointdec}; we refer to~\S\ref{sec:lin} for more details. On the other hand, the initial conditions
\begin{align*}u_i(x) = \re^{-\frac{x^2}{4d_i}} E_i, \qquad i \in \{1,\ldots,n\},\end{align*}
where $E_i \in \R^n$ is the $i$-th unit vector, satisfy $\F(u_i) \in W^{1,1}_{1+\alpha}(\R,\C^n)$ and yield the family of Gaussian solutions
\begin{align*}
u_i(x,t) = \frac{\re^{-\tfrac{(x+c_it)^2}{4d_i(1+t)}}}{\sqrt{1+t}} E_i, \qquad i \in \{1,\ldots,n\},
\end{align*}
to~\eqref{GRDlin} attaining the decay rates in~\eqref{temporaldec} and~\eqref{pointdec} up to the $\ln(2+t)$-factor in~\eqref{pointdec}. Thus, upon taking the nonlinearity in~\eqref{GRD} identically zero, one observes that the estimates in Theorem~\ref{mainresult1} are sharp up to the $\ln(2+t)$-factor in~\eqref{pointdec}. We strongly expect that this factor does not arise due to nonlinear effects, and can in fact be avoided by extending our analysis as outlined in the subsequent Remark~\ref{logrem}. However, in order to prevent this work from being overly technical, we refrain from doing so.

Thus, Theorem~\ref{mainresult1} shows that, if components propagate with different velocities in~\eqref{GRD}, then marginal mixed-terms and Burgers'-type terms do not affect the long-time behavior of small, algebraically localized initial data and solutions decay as predicted by the linear dynamics. We emphasize that, in general, marginal mixed-terms can be decisive for the long-term dynamics. For instance, in~\cite{ESCLEV} it was proved that every solution to
\begin{align*}
\begin{split}
\partial_t u_1 &= \partial_{xx} u_1 + u_1^{p_1}u_2^{q_1},\\
\partial_t u_2 &= \partial_{xx} u_2 + u_1^{p_2}u_2^{q_2},
\end{split} \qquad t \geq 0, x \in \R,
\end{align*}
having initial data $(u_{1,0},u_{2,0})$ satisfying $u_{1,0},u_{2,0}\geq 0$ and $u_{1,0}u_{2,0} \neq 0$ pointwise, blows up in finite time, if it holds $p_i,q_i \in \{1,2\}$ and $p_i+q_i = 3$ for $i = 1,2$.

Our proof of Theorem~\ref{mainresult1} relies on the analysis of~\eqref{GRD} in Fourier space. We exploit that a change to a \emph{co-moving frame} $\zeta_i = x + c_it$ in physical space corresponds to a multiplication with the exponential $\re^{-\ri c_i kt}$ in Fourier space. We multiply each component $v_i(k,t)$ of the Fourier transform $v(k,t) := (\F u)(k,t)$ of $u(x,t)$ with the appropriate exponential $\re^{-\ri c_i kt}$. This introduces \emph{oscillatory} factors in front of those critical nonlinear terms in~\eqref{GRD} which are of mixed or of Burgers' type. Thus, oscillatory integrals arise in Duhamel's formulation, whose decay properties can be exploited, as long as the velocities are different, by \emph{integrating by parts in time or in frequency}. To control derivatives with respect to the frequency in Fourier space, which appear as a result of integration by parts, we require that $u(x,t)$ lies in a polynomially weighted Sobolev space. Eventually, we are able to close a nonlinear iteration scheme in this space using $L^1$-$L^p$-estimates. In~\S\ref{sec:ill} we illustrate the main ideas behind our approach and sketch in a simple setting how to handle the most critical nonlinear terms.

\subsection{Relationship with space-time resonances method} \label{sec:spacetime}

It is interesting to compare our approach with the so-called \emph{space-time resonances method}~\cite{GER2,GMS,GMS2}, which has been developed by Germain, Masmoudi and Shatah to prove global existence of small initial data in nonlinear dispersive equations in $\R^d$, such as the nonlinear Schr\"odinger (NLS) or nonlinear wave equations. As in our diffusive setting, linear dispersion terms can force solutions to decay and spread, whereas nonlinear terms can cause solutions to grow and even blow up in finite time, cf.~\cite{GMS3} and \cite{IKEW}. Again one can distinguish between irrelevant nonlinearities, which correspond to a sufficiently high power, so that small solutions are governed by the linear dynamics, and relevant or marginal nonlinearities that can contribute to the large-time behavior of small initial data. For instance, in the case of the NLS equation, the critical threshold is given by the \emph{Strauss exponent}~\cite{STRAUSS}, which, as the Fujita exponent, decreases with the spatial dimension $d$.

The method of space-time resonances combines the strength of two earlier developed approaches to handle relevant or marginal nonlinearities in dispersive equations. More precisely, the space-time resonances method identifies the normal form method of Shatah~\cite{SHAT} as an integration by parts in time in the Duhamel formula in Fourier space, whereas integration by parts in frequency can be related to the vector field method~\cite{KLAI2,KLAI3} developed by Klainermann. As in our approach, integration by parts of oscillatory integrals might reveal additional decay, which can be exploited to close a nonlinear iteration scheme. However, the integration by parts in time or frequency can introduce singularities, so-called \emph{time and space resonances}, in Duhamel's formulation. The location of the time and space resonances is largely dependent on the interplay between the linearity and the nonlinear terms and some nonlinear terms might even cancel some of the singularities arising. Thus, the space-time resonances method has the potential to, at least partially, uncover the effect of a large class of relevant and marginal nonlinearities on the long-time dynamics in nonlinear dispersive equations. We refer to~\cite{GERM} to a short exposition of the key ideas of the space-time resonances method in a simple setting.

\subsection{Comparison with earlier result obtained with the method of pointwise estimates}

The effect of different velocities in RDA systems on the long-time dynamics of small initial data was, to the authors' best knowledge, first investigated in the recent paper~\cite{RSDV}. We compare Theorem~\ref{mainresult1} with the earlier results in~\cite{RSDV}, which were obtained with the method of pointwise estimates.

Perhaps the most apparent improvement is that Theorem~\ref{mainresult1} applies to multi-component RDA systems, whereas the results in~\cite{RSDV} are restricted to two components. As outlined in~\cite[Section 6]{RSDV}, it is still open whether the method of pointwise estimates can capture the effect of differences in velocities in general multi-component RDA systems. The number of terms in the spatio-temporal weight increases rapidly with the number of components, which complicates the pointwise analysis. In addition, qualitative new terms occur, which we were unable to control using pointwise estimates. In particular, we did not succeed in controlling mixed-terms in the $u_i$-equations with contributions of the $u_j$-th and $u_k$-th component for $i,j,k \in \{1,\ldots,n\}$ pairwise different, e.g.~a quadratic or cubic mixed-term of the form $u_ju_k$ or $u_j^2u_k$ in the $u_i$-equation.

A second difference between the results in~\cite{RSDV} and Theorem~\ref{mainresult1} is that the required localization on initial data in~\cite{RSDV} is stronger. Besides to exponentially localized initial data, the method of pointwise estimates can be applied to small, algebraically localized initial data $u_0 \in C^{0,\alpha}(\R,\R^n)$ satisfying $\|(1+|\cdot|)^r u_0\| \leq \delta \ll 1$ for $r \geq 3$. Such polynomial localization of initial data leads to the pointwise decay estimate
\begin{align} \left|u_i(x,t)\right| &\leq C\left[\frac{1}{\left(1 + |x+c_it| + \sqrt{t}\right)^{r}} + \frac{\re^{-\tfrac{(x+c_it)^2}{M(1+t)}}}{\sqrt{1+t}}\right], \qquad i = 1,2, \label{pointwise}\end{align}
on the components of the associated solution $u(x,t)$ to~\eqref{GRD} for $x \in \R$ and $t \geq 0$, where $C,M \geq 1$ are $x$- and $t$-independent constant. Thus, one finds that the part of the $i$-th component $u_i(x,t)$, exhibiting the slowest temporal decay, is in fact exponentially localized in the appropriate co-moving frame. The localization required in Theorem~\ref{mainresult1} corresponds to the case $r = 1$. It is interesting to note that the algebraic pointwise bound is then no longer exhibiting faster temporal decay than the exponential bound on the right-hand side of~\eqref{pointwise} and precisely coincides with the bound~\eqref{pointdec} established in our analysis.

A third difference is that regularity conditions on initial data are more relaxed in~\cite{RSDV}, which can be explained by the fact that all nonlinear terms with derivatives in~\cite{RSDV} are in divergence form, i.e.~the nonlinearity in~\cite{RSDV} takes the form
\begin{align}f(u,\partial_x u) = h(u) + \partial_x (g(u)), \label{divergenceform} \end{align}
whereas the nonlinearity in Theorem~\ref{mainresult1} can possess terms with derivatives which are not in divergence form. The derivative in~\eqref{divergenceform} can be moved onto the Green's function via integration by parts in Duhamel's formula, thus requiring less regular initial data to prove local existence of classical solutions. It therefore comes as no surprise that we `lost' one derivative, i.e.~we need $1+\alpha$ (fractional) derivatives in Theorem~\ref{mainresult1}, whereas $\alpha$ derivatives sufficed in~\cite{RSDV}.

Finally, we compare the class of allowable marginal and relevant nonlinear terms in~\cite{RSDV} and Theorem~\ref{mainresult1}. First of all, relevant mixed-terms, i.e.~products of the form $u_iu_j$ with $i \neq j$ in~\eqref{GRD}, cannot be handled by the analysis in the current paper, whereas those terms can be dealt with using the methods in~\cite{RSDV}. As outlined in the subsequent Remark~\ref{quadraticmix}, we expect that our approach could only handle such terms, if we have control over \emph{all} derivatives of each component of the solution in the appropriate co-moving frame in Fourier space, i.e.~over all $k$-derivatives of $\re^{-\ri c_i kt} (\F u)_i(k,t)$ for $i = 1,\ldots,n$. This would mean that, at least the slowest decaying part of solution $u(x,t)$, should have a stronger-than-polynomial localization in physical space in the appropriate co-moving frame. We emphasize that this is precisely the control we gain using pointwise estimates. Indeed, both for exponentially and algebraically localized initial data in~\cite{RSDV}, the slowest decaying part of the $i$-th component of the solution to~\eqref{GRD} is bounded by the drifting Gaussian
\begin{align*}
\frac{\re^{-\tfrac{(x+c_it)^2}{M(1+t)}}}{\sqrt{1+t}},
\end{align*}
which is exponentially localized in the appropriate co-moving frame $\zeta_i = x + c_i t$ for each fixed $t \geq 0$.

Second, the nonlinearity in~\cite{RSDV} is of the form~\eqref{divergenceform}, whereas the nonlinearity in Theorem~\ref{mainresult1} can contain terms with derivatives which are not in divergence form. In particular, marginal terms of the $u_i \partial_x(u_j)$ with $i \neq j$ can be handled by the analysis in the current paper. As mentioned before, the $x$-derivative in~\eqref{divergenceform} can be moved onto the Green's function via integration by parts in Duhamel's formulation. Consequently, it is not necessary to control $\partial_x u(x,t)$ in the nonlinear iteration in~\cite{RSDV}. Thus, by incorporating the derivative $\partial_x u(x,t)$ into the nonlinear iteration scheme in~\cite{RSDV}, we expect that there are no obstructions to handle marginal terms of the form $u_i \partial_x (u_j)$ with $i \neq j$, because differences in velocities between components can be exploited.

Third, Burgers'-type terms of the form $\partial_x (u_i^2)$ are only allowed in the $u_i$-equation in~\cite{RSDV}. In fact, a Burgers'-type term $\partial_x (u_i^2)$ in the $u_j$-equation for $i \neq j$, can interact with, a seemingly harmless, quadratic mixed-term. In fact, in the toy problem
\begin{align}
\begin{split}
\partial_t u_1&= d_1\partial_{xx} u_1 + c_1 \partial_x u_1 + \kappa u_1u_2 + \beta u_2^3,\\
\partial_t u_2&= d_2\partial_{xx} u_2 + c_2 \partial_x u_2 + \gamma \partial_x \left(u_2^2\right),
\end{split} \qquad t \geq 0, x \in \R, \label{CAS3}
\end{align}
with $d_1,d_2 > 0$ and $c_1,c_2 \in \R$ with $c_1 \neq c_2$, global existence of solutions with small, exponentially localized initial conditions is proved in~\cite[Theorem 1.5]{RSDV} under the condition that
\begin{align}\beta - \frac{\gamma \kappa}{c_2 - c_1} < 0,\label{expreta}\end{align}
is satisfied for the coefficients $\kappa,\beta,\gamma \in \R$. Thus, in the presence of quadratic mixed-terms, i.e.~in case $\kappa \neq 0$, the Burgers'-type term $\partial_x(u_2^2)$ can compensate for the `dangerous' cubic term $\beta u_2^3$. Indeed, in case $\kappa = 0$ and $\beta > 0$, all solutions to~\eqref{CAS3} with positive initial data blow up in finite time~\cite{HAYA}. On the other hand, even if $\beta = 0$, the expression~\eqref{expreta} suggests that the marginal term $\partial_x (u_2^2)$ might affect the long-time asymptotics. In Theorem~\ref{mainresult1} a Burgers'-type term $\partial_x (u_i^2)$ in the $u_j$-equation is allowed (even if $i \neq j$). This is not totally unexpected, as quadratic mixed-terms are absent in the nonlinearities in Theorem~\ref{mainresult1}. We note that a $\partial_x (u_i^2)$-term in the $u_j$-equation introduces, in case $i \neq j$, an oscillatory term when we multiply the Fourier transform of each component $(\F u)_i(k,t)$ with the appropriate factor $\re^{-\ri c_i kt}$. These oscillations can be exploited by integrating by parts in time in Duhamel's formula, see~\S\ref{sec:ill}. As outlined in~\S\ref{sec:time}, integration by parts in time can, in the dispersive setting, be linked to a normal form approach. It is therefore interesting to note that the effect of the Burgers'-type term $\partial_x (u_2^2)$ in~\eqref{CAS3} on the long-time dynamics is also in~\cite{RSDV} exposed via the normal form transform $\vt(\xt,t) = u_1(\xt-c_1t,t) + \frac{\gamma}{c_2 - c_1} u_2(\xt-c_1t,t)^2$.

\subsection{Set-up}

This paper is structured as follows. In~\S\ref{sec:ill}, we illustrate the main ideas behind our approach in a simple setting. Subsequently, we collect necessary local existence and uniqueness results of solutions to~\eqref{GRD} in~\S\ref{sec:loc}. The core of the paper entails the global analysis of solutions to~\eqref{GRD} with small initial data, which culminates in the proof of Theorem~\ref{mainresult1} in~\S\ref{sec:proofMIX}. Finally, we provide a future outlook and discuss open problems in~\S\ref{sec:discussion}.

\section{Illustration of the main ideas} \label{sec:ill}

This section provides a short introduction to the method employed in the proof of Theorem~\ref{mainresult1}. We illustrate in a simple setting how we control the most critical nonlinear terms exploiting oscillations that arise in Fourier space due to differences in velocities. We consider the toy model
\begin{align}
\begin{split}
\partial_t u_1&= d_1\partial_{xx}u_1 + c_1 \partial_x u_1 + \left(2 \pi u_1\right)^r u_2 + \left(2\pi\right)^{q-1} u_2^q,\\
\partial_t u_2&= d_2\partial_{xx}u_2 + c_2 \partial_x u_2 + 2 \pi \partial_x \left(u_1^2\right) + \left(2 \pi\right)^{q-1} u_2^q,
\end{split} \qquad t \geq 0, x \in \R, \label{TOY}
\end{align}
with $r \in \Nn_{\geq 2}$, $q \in \Nn_{\geq 4}$, $d_i > 0$ and $c_i \in \R$ with $c_1 \neq c_2$. The coefficients in~\eqref{TOY} are chosen for the sake of simplicity of exposition, but their precise values are unimportant in the further analysis. Indeed, applying the Fourier transform~\eqref{defFourier} to~\eqref{TOY} yields
\begin{align}
\begin{split}
\partial_t v_1 &= -k^2d_1v_1 + c_1 \ri k v_1 + v_1^{*r} * v_2 + v_2^{*q},\\
\partial_t v_2 &= -k^2d_2v_2 + c_2 \ri k v_2 + \ri k v_1^{*2} + v_2^{*q},
\end{split} \qquad t \geq 0, k \in \R, \label{TOY2}
\end{align}
where $*$ denotes the standard convolution product. Oscillatory exponentials arise when considering each component in~\eqref{TOY} in the appropriate co-moving frame. This corresponds to the coordinate change $w(k,t) = \re^{-c_1 \ri k t} v_1(k,t)$ and $z(k,t) = \re^{-c_2 \ri k t} v_2(k,t)$ in~\eqref{TOY2}. In the new coordinates system~\eqref{TOY2} reads
\begin{align}
\begin{split}
\partial_t w(k,t) &= -k^2d_1w(k,t) +  \int_\R \re^{(c_2 - c_1) \ri l t} w^{*r}(k-l,t) z(l,t) \de l + \re^{(c_2-c_1) \ri k t}z^{*q}(k,t),\\
\partial_t z(k,t) &= -k^2d_2z(k,t) +  \re^{(c_1-c_2) \ri k t} \ri k w^{*2}(k,t) + z^{*q}(k,t),
\end{split} \label{TOY3}
\end{align}
with $t \geq 0$ and $k \in \R$. We observe that, due to the difference in velocities, oscillatory exponentials arise in front of all nonlinear coupling terms, i.e.~in front of all terms with a $z$-contribution in the $w$-equation or terms with a $w$-contribution in the $z$-equation. The additional temporal decay induced by the oscillations can be revealed by integrating by parts in time or in frequency in the Duhamel formulation of~\eqref{TOY3}.

We take small initial data $(w_0,z_0) \in \smash{W^{1,1}_1}(\R,\C^2)$ to~\eqref{TOY3} satisfying $\|(w_0,z_0)\|_{\smash{W^{1,1}_1}} \leq \delta \ll 1$. We assume local existence and uniqueness of a continuous mild solution $(w(t),z(t))$ in $\smash{W^{1,1}_1}(\R,\C^2)$ to~\eqref{TOY3} with initial condition $(w_0,z_0)$ on some maximal time interval $[0,T)$ with $T \in (0,\infty]$, so that, if $T < \infty$, the $W^{1,1}_1$-norm of $(w(t),z(t))$ blows up as $t \uparrow T$. Thus, appropriate iterative estimates on the components
\begin{align*}
\left\||\cdot|^j \partial_k^m w(t)\right\|_1, \left\||\cdot|^j \partial_k^m z(t)\right\|_1, \qquad j,m = 0,1,
\end{align*}
of the $W^{1,1}_1$-norm of the solution prove that such blow-up cannot occur and yield global existence and decay, see~\S\ref{sec:planM} for more details. Such estimates can be obtained through the Duhamel formulation (or variation of constants formula) corresponding to~\eqref{TOY3}, which is given by
\begin{align}
\begin{split}
w(k,t) &= \re^{-d_1 k^2 t} w_0(k) + \int_0^t \int_\R \re^{-k^2 d_1(t-s) + (c_2 - c_1) \ri l s} w^{*r}(k-l,s) z(l,s) \de l \de s\\
&\qquad \qquad \qquad \qquad \qquad \qquad + \int_0^t \re^{-k^2 d_1(t-s) + (c_2-c_1)\ri k s} z^{*q}(k,s) \de s,\\
z(k,t) &= \re^{-d_2 k^2 t} z_0(k) + \int_0^t \ri k \re^{-k^2 d_2(t-s) + (c_1-c_2) \ri k s} w^{*2}(k,s) \de s+ \int_0^t \re^{-k^2 d_2(t-s)} z^{*q}(k,s) \de s,
\end{split}
\label{Duhamelclas}
\end{align}
for $k \in \R$ and $t \in [0,T)$.

It is not hard, cf.~\S\ref{sec:lin}, to establish the estimate
\begin{align*}
\int_\R \left|k^j \partial_k^m \re^{-d_1 k^2 s} w(k)\right| \de k \leq C\frac{\|w\|_{\smash{W^{1,1}_1}}}{(1+s)^{\frac{1+j-m}{2}}}, \qquad w \in W^{1,1}_1(\R,\C^2),\, j,m = 0,1, \, s \geq 0,
\end{align*}
where $C\geq1$ is some $s$-independent constant. Therefore, if the nonlinear terms in~\eqref{TOY3} were absent, the solution $(w(s),z(s))$ would decay as
\begin{align} \left\||\cdot|^j \partial_k^m w(s)\right\|_1, \left\||\cdot|^j \partial_k^m z(s)\right\|_1 \leq \frac{C\delta}{(1+s)^{\frac{1+j-m}{2}}}, \qquad j,m = 0,1, \, s \geq 0. \label{assump} \end{align}

The general idea of a nonlinear iteration scheme is to employ the bounds~\eqref{assump} on the linear terms in~\eqref{Duhamelclas} to obtain estimates on the nonlinear terms in~\eqref{Duhamelclas}. To illustrate this principle, let us bound the last integral in the $w$-component of~\eqref{Duhamelclas}, which corresponds to an irrelevant nonlinearity. Thus, take $t \in [0,T)$ and assume~\eqref{assump} holds for all $s \in [0,t)$. Using Young's convolution inequality, the fact that $W^{1,1}(\R,\C)$ is continuously embedded in $L^\infty(\R,\C)$ by the fundamental theorem of calculus and the fact that $q \geq 4$, we obtain for $j = 0,1$ the estimate
\begin{align}
\begin{split}
&\int_\R \left|k^j \partial_k \int_0^t \re^{-k^2 d_1(t-s) + (c_2-c_1) \ri k s} z^{*q}(k,s) \de s\right| \de k \\
&\ \ \leq C\left(\int_0^t \int_\R \left|k^{j+1} (t-s) \re^{-k^2 d_1(t-s)} z^{*q}(k,s)\right| \de k \de s + \int_0^t \int_\R \left|k^j s \re^{-k^2 d_1(t-s)} z^{*q}(k,s) \right|\de k \de s\right.\\
& \qquad \qquad \qquad \left. +\ \int_0^t \int_\R \left|k^j \re^{-k^2 d_1(t-s)} \partial_k \left(z^{*q}(k,s)\right) \right|\de k \de s\right)\\
&\ \ \leq C\left(\int_0^t \int_\R k^{j+1} (t-s) \re^{-k^2 d_1(t-s)} \de k \left\|z(s)\right\|_\infty\left\|z(s)\right\|_1^{q-1} \de s \phantom{ \left(\int_\R \left|k^j \re^{-k^2 d_1(t-s)}\right|^2\de k\right)^{\frac{1}{2}}}\right. \\
&\qquad \qquad \qquad \left. +\, \int_0^t \left(\int_\R \left|k^j \re^{-k^2 d_1(t-s)}\right|^2\de k\right)^{\frac{1}{2}} \left\|z(s)\right\|^{\frac{1}{2}}_\infty \left( s \left\|z(s)\right\|_1^{q-\frac{1}{2}} + \left\|z(s)\right\|_1^{q-\frac{3}{2}} \left\|\partial_k z(s)\right\|_1\right) \de s \right)\\
&\ \ \leq C\delta^2 \left(\int_0^t \frac{1}{(t-s)^{\frac{j}{2}} (1+s)^{\frac{q-1}{2}}} \de s + \int_0^t \frac{1}{(t-s)^{\frac{1+2j}{4}} (1+s)^{\frac{2q-5}{4}}} \de s\right) \leq C\frac{\delta^2}{(1+t)^{\frac{j}{2}}},
\end{split} \label{irrill2}
\end{align}
and
\begin{align}
\begin{split}
&\int_\R \left|k^j \int_0^t \re^{-k^2 d_1(t-s) + (c_2-c_1) \ri k s} z^{*q}(k,s) \de s\right| \de k\\
&\qquad \leq C\left(\int_0^{\frac{t}{2}} \int_\R k^j \re^{-k^2 d_1(t-s)} \de k \left\|z(s)\right\|_\infty\left\|z(s)\right\|_1^{q-1} \de s + \int_{\frac{t}{2}}^t \sup_{k \in \R} \left(k^j \re^{-k^2 d_1(t-s)}\right) \left\|z(s)\right\|_1^q \de s\right)\\
&\qquad \leq C\delta^2 \left(\int_0^{\frac{t}{2}} \frac{1}{(t-s)^{\frac{1+j}{2}} (1+s)^{\frac{q-1}{2}}} \de s + \int_{\frac{t}{2}}^t \frac{1}{(t-s)^{\frac{j}{2}} (1+s)^{\frac{q}{2}}} \de s \right) \leq C\frac{\delta^2}{(1+t)^{\frac{1+j}{2}}}.
\end{split} \label{irrill1}
\end{align}
where we denote by $C \geq 1$ any $t$-independent constant. Hence, we conclude that the last integral in the $z$-component exhibits those decay properties as one would expect from the linear dynamics~\eqref{assump}.

To close the nonlinear iteration scheme, we need to obtain similar estimates on the other nonlinear terms in the Duhamel formulation~\eqref{Duhamelclas}. To obtain estimate~\eqref{irrill1} one readily observes that it was crucial that $q > 3$, whereas for estimate~\eqref{irrill2} we needed $q \geq 4$. So, we cannot expect that a similar procedure works to bound those integrals in~\eqref{Duhamelclas}, which correspond to the marginal nonlinear terms $u_1^ru_2$ and $\partial_x (u_2^2)$ in~\eqref{TOY}. We explain below how to bound such integrals by either integrating by parts in time or in frequency.

\subsection{Integration by parts in frequency} \label{sec:illfreq}

Take $t \in [0,T)$ and let us consider the integral
\begin{align*}
I_m(k,t) := \int_0^t \int_\R \re^{-k^2 d_1(t-s) + (c_2 - c_1) \ri l s} &w^{*r}(k-l,s) z(l,s) \de l \de s, \qquad k \in \R,
\end{align*}
in~\eqref{Duhamelclas} corresponding to the marginal mixed-term $u_1^r u_2$ in the $u_1$-equation in~\eqref{TOY}. To avoid singularities in time, we split the domain of integration in a part from $0$ to $1$, which can be bounded as in~\eqref{irrill1}, and a more problematic part from $1$ to $t$. To gain additional temporal decay in the second integral for $t \geq 2$, we integrate by parts in frequency and use that $w(s), z(s) \in \smash{W^{1,1}_1(\R,\C)}$ are localized for $s \in [0,t]$, to obtain
\begin{align}
\begin{split}
\int_1^t \int_\R &\re^{-k^2 d_1(t-s) + (c_2 - c_1) \ri l s} w^{*r}(k-l,s) z(l,s) \de l \de s\\
&\qquad = -\int_1^t \int_\R \frac{\re^{-k^2 d_1(t-s) + (c_2 - c_1) \ri l s} \partial_l\left(w^{*r}(k-l,s) z(l,s)\right)}{(c_2-c_1)\ri s} \de l \de s,
\end{split} \qquad k \in \R.
\label{intspace0}
\end{align}
By assuming~\eqref{assump}, identity~\eqref{intspace0} leads for $t \geq 2$ and $j = 0,1$ to the bound
\begin{align*}
&\int_\R \left|k^j \int_1^t \int_\R \re^{-k^2 d_1(t-s) + (c_2 - c_1) \ri l s} w^{*r}(k-l,s) z(l,s) \de l \de s\right| \de k\\
&\qquad \leq C\left(\int_1^{\frac{t}{2}} \int_\R k^j \re^{-k^2 d_1(t-s)} \de k \, s^{-1} \left\|w(s)\right\|_\infty\|w(s)\|_1^{r-2}\left(\left\|w(s)\right\|_1\left\|\partial_k z(s)\right\|_1 + \left\|\partial_k w(s)\right\|_1\left\|z(s)\right\|_1\right) \de s\right.\\
&\qquad \qquad \qquad \left. +\, \int_{\frac{t}{2}}^t \sup_{k \in \R} \left(k^j \re^{-k^2 d_1(t-s)}\right) s^{-1} \left\|w(s)\right\|_1^{r-1} \left(\left\|w(s)\right\|_1\left\|\partial_k z(s)\right\|_1 + \left\|\partial_k w(s)\right\|_1\left\|z(s)\right\|_1\right) \de s\right)\\
&\qquad \leq C\left(\int_1^{\frac{t}{2}} \frac{1}{s(t-s)^{\frac{1+j}{2}} (1+s)^{\frac{r-1}{2}}} \de s + \int_{\frac{t}{2}}^t \frac{1}{s(t-s)^{\frac{j}{2}} (1+s)^{\frac{r}{2}}} \de s\right) \leq C\frac{\delta^2}{(1+t)^{\frac{1+j}{2}}}.
\end{align*}
Short-time bounds on $I_m(t)$ for $t \leq 2$ can then be established similarly as in~\eqref{irrill1}. In the bounds on the $k$-derivative $\partial_k I_m(t)$ the additional temporal decay obtained by integrating by parts in frequency can also be exploited. However, integrating by parts the term
\begin{align}
I_r(k,t) := \int_0^t \int_\R \re^{-k^2 d_1(t-s) + (c_2 - c_1) \ri l s} \partial_k\left(w^{*r}(k-l,s)\right) z(l,s) \de l \de s, \qquad k \in \R, \label{obstr}
\end{align}
arising in $\partial_k I_m(t)$, leads to a double derivative $\partial_k^2 w(\cdot,s)$ in the convolution product, whose $L^p$-norm is not bounded by the $W^{1,1}_1$-norm of $w(s)$ for any $p \in [1,\infty]$. Instead, we bound~\eqref{obstr} directly and avoid integrating by parts, which leads, as in~\eqref{irrill2}, for $j = 0,1$ to the bound
\begin{align}
\left\||\cdot|^j I_r(t)\right\|_1 \leq C\delta^2 \int_0^t \frac{1}{(t-s)^{\frac{1+2j}{4}} (1+s)^{\frac{2r-1}{4}}} \de s \leq C\frac{\delta^2}{(1+t)^{\frac{j}{2}}}, \label{irrill3}
\end{align}

\begin{remark}\label{quadraticmix}
{\upshape
For the last inequality in~\eqref{irrill3} to hold, and thus to close the nonlinear iteration scheme, we observe that it is crucial that $r \geq 2$. This shows that quadratic mixed-term, i.e.~the case $r = 1$, cannot be handled by the method presented in this paper. The desired bounds on quadratic mixed-terms would require integrating by parts in frequency once again in~\eqref{obstr} in order to obtain sufficient decay in $s$, which would lead to the double derivative $\partial_k^2 w(k-l,s)$. At first sight, controlling the double derivative $\partial_k^2 w(k,s)$ in the nonlinear iteration scheme seems a solution to this obstruction. However, a similar problem then occurs in bounding $\partial_k^2 I_m(t)$, which would then, after integrating by parts in frequency, require control over the third derivative $\partial_k^3 w(k-l,s)$. In fact, we would need control over \emph{all} $k$-derivatives of $w(k,s)$ for our approach to work for quadratic mixed-terms, which would complicate the analysis and require stronger-than-polynomially localized initial data; see also~\S\ref{sec:discussion}.}
\end{remark}

\subsection{Integration by parts in time} \label{sec:time}

Take $t \in [0,T)$ and let us consider the integral
\begin{align*}
I_b(k,t) := \int_0^t \ri k \re^{-k^2 d_2(t-s) + (c_1-c_2) \ri k s} w^{*2}(k,s) \de s = 2\int_0^t \int_\R \ri \re^{-k^2 d_2(t-s) + (c_1-c_2) \ri k s} w(k-l,s) \, l w(l,s) \de l\de s,
\end{align*}
in~\eqref{Duhamelclas}, corresponding to the Burgers'-type coupling $\partial_x (u_1^2)$ in the $u_2$-equation in~\eqref{TOY}. Although $\partial_x (u_1^2)$ is a marginal nonlinearity, we can move the spatial derivative onto the semigroup and proceed as in~\eqref{irrill1} to establish for $j = 0,1$ the desired estimate
\begin{align*}
\left\||\cdot|^j I_b(t)\right\|_1 \leq C\delta^2 \left(\int_0^{\frac{t}{2}} \frac{1}{(t-s)^{1+\frac{j}{2}} \sqrt{1+s} } \de s + \int_{\frac{t}{2}}^t \frac{1}{(t-s)^{\frac{j}{2}} (1+s)^{\frac{3}{2}}} \de s\right) \leq C\frac{\eta(t)^2}{(1+t)^{\frac{1+j}{2}}}.
\end{align*}
However, the $k$-derivative of $I_b(k,t)$ contains the term
\begin{align*}
J_b(k,t) := \int_0^t k (c_2-c_1) \re^{-k^2 d_2(t-s) + (c_1-c_2) \ri k s} s \, w^{*2}(k,s) \de s, \qquad k \in \R,
\end{align*}
which cannot be bounded as in~\eqref{irrill2} (we would need $w^{*3}(k,s)$ instead of $w^{*2}(k,s)$ to obtain such a bound). To establish additional temporal decay, we integrate by parts in time and find
\begin{align}
J_b(k,t) = \frac{\ri(c_2-c_1)}{d_2 k + (c_1 - c_2)\ri} \left(t\, \re^{(c_1 - c_2)\ri k t} w^{*2}(k,t) - \int_0^t \re^{-k^2 d_2(t-s) + (c_1-c_2) \ri k s} \partial_s \left(s\, w^{*2}(k,s)\right) \de s\right), \label{inttime0}
\end{align}
with $k \in \R$. We emphasize that no singularities are introduced due to the special divergence form of the Burgers'-type term, which vanishes at frequency $k = 0$ in Fourier space, i.e.~in the language of Germain, Masmoudi and Shatah, see~\S\ref{sec:spacetime}, the time resonance at $k = 0$ is canceled. We can now replace the temporal derivative~$\partial_s w^{*2}(k,s)$ in~\eqref{inttime0} using the $w$-equation in~\eqref{TOY3}. The remaining terms in~\eqref{inttime0} can now be bounded more or less in the standard way, cf.~\eqref{irrill1} and~\eqref{irrill2}. We refer to~\S\ref{sec:Burgers} for further details.

\section{Local existence and uniqueness} \label{sec:loc}

Local existence and uniqueness of classical solutions to semilinear parabolic equations is well-esta-blished for bounded, H\"older continuous initial conditions, see for instance~\cite{LUN}. We collect the necessary results for reaction-diffusion-advection systems from~\cite[Section 11.3]{ZUH}, which were obtained using the so-called parametrix method. Subsequently, we connect these results to our global estimates by establishing local control on the Fourier transform of solutions to~\eqref{GRD} in the weighted Sobolev space $W^{1,1}_1(\R,\C^n)$.

First, we observe that the method in~\cite[Section 11.3]{ZUH} is applicable to prove local existence and uniqueness of solutions to~\eqref{GRD} in the weighted Sobolev space
\begin{align*} W_1^{1,\infty}(\R,\R^n) := \left\{u \in W^{1,\infty}(\R,\R^n) : \|u\|_{\smash{W^{1,\infty}_1}} < \infty\right\},\end{align*}
which is equipped with the norm
\begin{align*} \|u\|_{W^{1,\infty}_1} = \|u\|_\infty + \|\partial_x u\|_\infty + \||\cdot| u\|_\infty + \||\cdot| \partial_x u\|_\infty.\end{align*}
Indeed, if $u(x,t)$ solves~\eqref{GRD}, then the function $U(x,t) = (u(x,t),x u(x,t))$ solves again a RDA system with H\"older continuous coefficients and sufficiently smooth nonlinearities. Subsequently, we employ a standard, but not readily available, regularity argument to prove that the Fourier transform $(\F u)(t)$ of the obtained local solution to~\eqref{GRD} in $W^{1,\infty}_1(\R,\R^n)$ exists in $W^{1,1}_1(\R,\C^n)$ and is continuous with respect to time.

All in all, we establish the following local existence result.

\begin{proposition} \label{proplocal}
Let $\alpha > 0$. Suppose that the coefficients in~\eqref{GRD} satisfy $d_i > 0$ and it holds $f \in C^{2,\alpha}(\R^n \times \R^n,\R^n)$ with $f(0,0) = Df(0,0) = 0$. Take $v_0 \in W_{1+\alpha}^{1,1}(\R,\C^n)$ satisfying the reality condition $v_0(-k) = \smash{\overline{v_0(k)}}$ for each $k \in \R$. Then, there exists $T \in (0,\infty]$ such that we have a unique classical solution $u_\ast \in C^{1,\frac{\alpha}{2}}\big((0,T),C_b^{3,\alpha}(\R,\R^n)\big)$ to~\eqref{GRD} with initial condition $u_\ast(0) = \F^{-1}(v_0)$. In addition, $v_\ast \colon [0,T) \to W^{1,1}_1(\R,\C^n)\big)$ given by $v_\ast(t) = \F(u_\ast(t))$ is continuous and $T > 0$ is maximal in the sense that, if it holds $T < \infty$, then we have
\begin{align}\limsup_{t \uparrow T} \left\|v_\ast(t)\right\|_{W^{1,1}_1} = \infty.\label{blowuplinfty}\end{align}
\end{proposition}
\begin{proof}
Upon setting $w(x) = \partial_x u(x)$, $p(x) = xu(x)$ and $q(x) = x\partial_x u(x)$, we rewrite~\eqref{GRD} as the $4n$-component system
\begin{align}
\begin{split}
\partial_t u&= \D \partial_{xx} u + \Ce \partial_x u + f\left(u,w\right), \\
\partial_t w&= \D \partial_{xx} w + \Ce \partial_x w + \partial_x \left(f(u,w)\right), \\
\partial_t p&= \D \partial_{xx} p + \Ce \partial_x p - 2 \D \partial_x u - \Ce u + x f\left(u,w\right), \\
\partial_t q&= \D \partial_{xx} q + \Ce \partial_x q - 2 \D \partial_x w - \Ce w + \partial_x \left(x f(u,w)\right) - f(u,w),
\end{split}
\qquad t \geq 0, x \in \R, \label{GRD2}
\end{align}
so that all nonlinear terms with derivatives are in divergence form and the coefficients and nonlinearity are $C^{2,\alpha}$-functions of $x$ and $(u,w,p,q)$. The relevant initial condition to~\eqref{GRD2} is
\begin{align*}U_0 := \left(u_0, \partial_x u_0, \rho u_0, \rho \partial_x u_0\right),\end{align*}
with $u_0 := \F^{-1}(v_0)$ and $\rho \colon \R \to \R$ given by $\rho(x) = x$. By~\cite[Proposition 5.2]{CAFF} there exists a constant $C \geq 1$ such that
\begin{align*} \|\rho^j u_0\|_{C^{1,\alpha}} &\leq C\left(\left\|(-\Delta)^{\frac{1+\alpha}{2}}\left(\rho^j u_0\right)\right\|_\infty + \|\rho^j u_0\|_\infty\right) \\
&\leq C\left(\left\||\cdot|^{1+\alpha} \partial_k^j v_0\right\|_1 + \left\|\partial_k^j v_0\right\|_1\right) \leq C\|v_0\|_{W^{1,1}_{1 + \alpha}},
\end{align*}
for $j = 0,1$. So, it holds $U_0 \in C_b^{0,\alpha}(\R,\R^{4n})$.

Thus, by~\cite[Corollary 11.4]{ZUH} and its proof, there exists a unique solution
\begin{align} U_\ast(x,t) = (u_\ast,w_\ast,p_\ast,q_\ast)(x,t), \quad U_\ast \in C^{0,\frac{\alpha}{2}}\big([0,T),C_b^{0,\alpha}(\R,\R^{4n})\big) \cap C^{1,\frac{\alpha}{2}}\big((0,T),C_b^{2,\alpha}(\R,\R^{4n})\big), \label{regularity}\end{align}
to~\eqref{GRD2} on a maximal interval $[0,T)$, with $T \in (0,\infty]$, having initial condition $U_0 \in C_b^{0,\alpha}(\R,\R^{4n})$. It is not difficult to verify that, by uniqueness of solutions, it must hold $w_\ast(x,t) = \partial_x u_\ast(x,t), p_\ast(x,t) = x u_\ast(x,t)$ and $q_\ast(x,t) = x w_\ast(x,t)$ for each $x \in \R$ and $t \in [0,T)$. So, on the one hand,~\eqref{regularity} entails that we have established a classical solution $u_\ast \in C^{1,\frac{\alpha}{2}}\big((0,T),\smash{C_b^{3,\alpha}}(\R,\R^n)\big)$ to~\eqref{GRD} with initial condition $u_\ast(0) = u_0$. On the other hand,~\eqref{regularity} also implies $u_\ast \in C^0\big([0,T),\smash{W_1^{1,\infty}}(\R,\R^n)\big)$.

Note that $\smash{W^{1,\infty}_1}(\R,\R^n)$ is continuously embedded in the Sobolev space $H^1(\R,\R^n)$. Hence, the Fourier transform maps $\smash{W^{1,\infty}_1}(\R,\R^n)$ continuously into the weighted $L^2$-space
\begin{align*}L^2_1(\R,\C^n) := \left\{v \in L^2(\R,\C^n) : \|v\|_{\smash{L_1^2}} < \infty\right\},\end{align*}
which is equipped with the norm $\|v\|_{\smash{L^2_1}} = \|(1+|\cdot|^2)^{1/2} v\|_2$, where $\|\cdot\|_2$ denotes the $L^2$-norm. The range of $\F$ in $L^2_1(\R,\C^n)$ is given by the subspace
\begin{align*} X := \left\{\F(u) \in L^2_1(\R,\C^n) : u \in W^{1,\infty}_1(\R,\R^n)\right\}. \end{align*}
Thus, the map $v_\ast \colon [0,T) \to X$ given by $v_\ast(t) = \F(u_\ast(t))$ is well-defined.

Fix $t \in [0,T)$. We prove that $v_\ast(t)$ lies in fact in $W^{1,1}_1(\R,\C^n)$. We denote by $C \geq 1$ any constant, which is only dependent on $n, \D$ and $\Ce$. We integrate~\eqref{GRD} and apply the Fourier transform to arrive at the Duhamel formulation:
\begin{align} v_\ast(k,t) = \re^{\left(-k^2 \D + \Ce \ri k\right)t} v_0(k) + \int_0^{t} \re^{\left(-k^2 \D + \Ce \ri k\right)(t-s)} \N(v_\ast(s))(k) \de s, \qquad k \in \R, \label{Duhamel1}\end{align}
where $\N \colon X \to H^1(\R,\C^n)$ is the nonlinear operator
\begin{align*} \N(v)(k) = \F \left[f\left(\F^{-1} v, \partial_x \F^{-1} v\right)\right](k).\end{align*}
We note that $\N$ is well-defined, because, by Taylor's Theorem and the fact that $f(0,0) = Df(0,0) = 0$, it holds
\begin{align}
\begin{split}
\left\|\partial_k^j \N(\F(u))\right\|_2 &\leq C\left\||\cdot|^j f\left(u, \partial_x u\right)\right\|_2 \\
&\leq C\left\|u\right\|_{H^1} \left\|u\right\|_{W^{1,\infty}_1} \sup_{\begin{smallmatrix} (v,w) \in \R^n \times \R^n \\ \|v\|,\|w\| \leq \|u\|_{\smash{W^{1,\infty}_1}}\end{smallmatrix}} \left\|D^2f\left(v,w\right)\right\|, \end{split}\label{Taylors}
\end{align}
for $u \in W^{1,\infty}_1(\R,\R^n)$ and $j = 0,1$. In fact, since $u_\ast \colon [0,T) \to W_1^{1,\infty}(\R,\R^n)$ is continuous, the nonlinear map $N_\ast \colon [0,t] \to H^1(\R,\C^n)$ given by $N_\ast(s)(k) = \N(v_\ast(s))(k)$ is bounded. On the one hand, for $j = 0,1$ we have
\begin{align}
\begin{split}
\int_\R \left\|(1+|k|) \partial_k^j \re^{\left(-k^2 \D + \Ce \ri k\right)t} v_0(k)\right\| \de k
&\leq C\left(\left\|(1+|\cdot|)\partial_k^j v_0\right\|_1 + \left(\sqrt{t} + t\right)\|(1+|\cdot|)v_0\|_1\right)\\
&\leq C\left(1+t\right)\|v_0\|_{\smash{W^{1,1}_1}}.\end{split} \label{linearLoc}
\end{align}
On the other hand, given a bounded map $N \colon [0,t] \to H^1(\R,\C^n)$ and $j = 0,1$, we use H\"older's inequality and the fact that $x \mapsto p(x)\re^{-d x^2}$ is bounded on $\R$ for any polynomial $p \colon \R \to \R$ and $d > 0$ to yield
\begin{align}
\begin{split}
\int_\R &\int_0^{t} \left\|(1+|k|) \partial_k^j \re^{\left(-k^2 \D + \Ce \ri k\right)(t-s)} N(s)(k)\right\| \de s \de k\\
&\quad  \leq C \int_\R \int_0^{t} (1+|k|)\left\|\re^{\left(-k^2 \D\right)(t-s)}\right\| \left[(|k|(t-s) + (t-s))^j \left\|N(s)(k)\right\| + \left\|\partial_k N(s)(k)\right\|^j\right] \de s \de k\\
&\quad  \leq C \sup_{s \in [0,t]} \|N(s)\|_{H^1} \int_0^t \left(1 + t - s\right)\left\|(1+|\cdot|)\re^{-\frac{1}{2}(\cdot)^2 \D(t-s)}\right\|_2 \de s\\
&\quad  \leq C \sup_{s \in [0,t]} \|N(s)\|_{H^1} \int_0^t \frac{1 + t - s }{(t-s)^{\frac{1}{4}}} \left(1 + \frac{1}{\sqrt{t-s}}\right) \de s \\
&\quad  \leq C\left(1 + t^{\frac{7}{4}}\right) \sup_{s \in [0,t]} \|N(s)\|_{H^1}.
\end{split} \label{nonlinearLoc}
\end{align}
Hence, by~\eqref{Duhamel1},~\eqref{Taylors},~\eqref{linearLoc} and~\eqref{nonlinearLoc}, it holds $v_\ast(t) \in W^{1,1}_1(\R,\C^n)$ for each $t \in [0,T)$.

Next, we prove that $v_\ast \colon [0,T) \to W^{1,1}_1(\R,\C^n)$ is continuous. Fix $T_0 \in (0,T)$. It is sufficient to prove that $v_\ast$ is H\"older continuous on $[0,T_0]$. Take $s,t \in [0,T_0]$ with $s \leq t$. By~\eqref{Duhamel1} we have
\begin{align}
\begin{split}
v_\ast(k,t) - v_\ast(k,s) &= \int_s^t \left(-k^2 \D + \Ce \ri k\right)\re^{\left(-k^2 \D + \Ce \ri k\right)r} \de r v_0(k) + \int_s^t \re^{\left(-k^2 \D + \Ce \ri k\right)(t-r)} \N(v_\ast(r))(k) \de r\\
&\qquad  + \int_0^s \int_{s-r}^{t-r} \left(-k^2 \D + \Ce \ri k\right)\re^{\left(-k^2 \D + \Ce \ri k\right)\tau} \de \tau \N(v_\ast(r))(k) \de r,\end{split} \label{Duhamel2}\end{align}
for $k \in \R$. We denote by $C \geq 1$ any constant, which is only dependent on $n, \D, \Ce$ and $T_0$. On the one hand, for $j = 0,1$ we have
\begin{align}
\begin{split}
&\int_\R \int_s^t \left\|(1+|k|) \partial_k^j \left(-k^2 \D + \Ce \ri k\right) \re^{\left(-k^2 \D + \Ce \ri k\right)r} v_0(k)\right\| \de r\de k\\
&\qquad \leq C\int_s^t \left(\left\|(1+|\cdot|)^{1+\alpha} \partial_k^j v_0\right\|_1 \left\|(1 + |\cdot|)^{2-\alpha} \re^{-(\cdot)^2 \D r}\right\|_\infty\right. \\
&\qquad \qquad \qquad \qquad + \left. \left\|(1+|\cdot|) v_0\right\|_1 \left\|(1 + |\cdot|) \left(1 + (1 + |\cdot|)^2 r\right) \re^{-(\cdot)^2 \D r}\right\|_\infty \right)\de r\\
&\qquad \leq C  \|v_0\|_{\smash{W^{1,1}_{1+\alpha}}} \int_s^t r^{\frac{\alpha}{2}-1} \de r
\leq C  \|v_0\|_{\smash{W^{1,1}_{1+\alpha}}} \left(t^{\frac{\alpha}{2}} - s^{\frac{\alpha}{2}}\right), \end{split} \label{linearLoc1}
\end{align}
where we use $r \leq T_0$ for $r \in [s,t]$ to bound the integrand. On the other hand, given a bounded map $N \colon [0,T_0] \to H^1(\R,\C^n)$, we establish, as in~\eqref{nonlinearLoc}, the estimate
\begin{align}
\begin{split}
\int_\R \int_s^{t} &\left\|(1+|k|) \partial_k^j \re^{\left(-k^2 \D + \Ce \ri k\right)(t-r)} N(r)(k)\right\| \de r \de k\\
& \qquad \qquad \leq C \sup_{r \in [0,T_0]} \|N(r)\|_{H^1} \int_s^t \left(1 + t - r\right)\left\|(1+|\cdot|)\re^{-\frac{1}{2}(\cdot)^2 \D(t-r)}\right\|_2 \de r\\
& \qquad \qquad \leq C \sup_{r \in [0,T_0]} \|N(r)\|_{H^1} \int_s^t (t-r)^{-\frac{3}{4}} \de r \leq C (t-s)^{\frac{1}{4}} \sup_{r \in [0,T_0]} \|N(r)\|_{H^1},
\end{split} \label{nonlinearLoc1}
\end{align}
where we use $t-r \leq T_0$ for $r \in [s,t]$ to bound the integrand. Similarly, for $j = 0,1$ we arrive at
\begin{align}
\begin{split}
\int_\R \int_0^s \int_{s-r}^{t-r} &\left\|(1+|k|) \partial_k^j \left(-k^2 \D + \Ce \ri k\right)\re^{\left(-k^2 \D + \Ce \ri k\right)\tau} N(r)(k)\right\| \de \tau \de r \de k\\
& \qquad \qquad \leq C \sup_{r \in [0,T_0]} \|N(r)\|_{H^1} \int_0^s \int_{s-r}^{t-r}  \left\|(1+|\cdot|)^3\left(1+|\cdot|\tau\right)\re^{-(\cdot)^2 \D\tau}\right\|_2 \de \tau \de r\\
& \qquad \qquad \leq C \sup_{r \in [0,T_0]} \|N(r)\|_{H^1} \int_0^s \int_{s-r}^{t-r}  \tau^{-\frac{7}{4}} \de \tau \de r \leq C \left(t^{\frac{1}{4}} - s^{\frac{1}{4}}\right) \sup_{r \in [0,T_0]} \|N(r)\|_{H^1},
\end{split} \label{nonlinearLoc2}
\end{align}
where we use $\tau \leq T_0$ for $\tau \in [s,t]$ to bound the integrand. By~\eqref{Taylors},~\eqref{Duhamel2},~\eqref{linearLoc1},~\eqref{nonlinearLoc1} and~\eqref{nonlinearLoc2} the function $v_\ast \colon [0,T_0] \to W^{1,1}_1(\R,\C^n)$ is H\"older continuous for each $T_0 \in [0,T)$. Hence, it holds $v_\ast \in C^0\big([0,T), W^{1,1}_1(\R,\C^n)\big)$.

Finally, assume by contradiction that $T < \infty$ and~\eqref{blowuplinfty} is false, so that $t \mapsto \|v_\ast(t)\|_{\smash{W^{1,1}_1}}$ is bounded on $[0,T)$. Then, since the inverse Fourier transform maps $\smash{W^{1,1}_1(\R,\C^n)}$ continuously into $\smash{W^{1,\infty}_1(\R,\C^n)}$ and we have $v_\ast(t) = \F(u_\ast(t))$ for each $t \in [0,T)$, we find that the solution $U_\ast(x,t)$ to~\eqref{GRD} is bounded on $[0,T) \times \R$. As $f(0,0) = 0$, one observes that $U_\ast(x,t)$ satisfies the parabolic linear system
\begin{align}
\partial_t U = D \partial_{xx} U + \partial_x (G(x,t) U) + F(x,t) U, \label{linsys}
\end{align}
with $D := \text{diag}(\D,\D,\D,\D)$ and $F, G \colon [0,T) \times \R \to \R^{4n \times 4n}$ are given by
\begin{align*}
G(x,t) &:= \int_0^1 \begin{pmatrix} \Ce & 0 & 0 & 0 \\ f^u(\gamma,x,t) & f^w(\gamma,x,t) + \Ce & 0 & 0 \\ - 2\D & 0 & \Ce & 0 \\ x f^u(\gamma,x,t) & x f^w(\gamma,x,t) - 2\D & 0 & \Ce \end{pmatrix} \de\gamma,\\  F(x,t) &:= \int_0^1 \begin{pmatrix} f^u(\gamma,x,t) & f^w(\gamma,x,t) & 0 & 0 \\ 0 & 0 & 0 & 0 \\ x f^u(\gamma,x,t) - \Ce & x f^w(\gamma,x,t) & 0 & 0 \\ -f^u(\gamma,x,t) & -f^w(\gamma,x,t) - \Ce & 0 & 0 \end{pmatrix} \de\gamma,
\end{align*}
where we denote
\begin{align*}
f^u(\gamma,x,t) := \partial_u f(\gamma u_\ast(x,t),\gamma w_\ast(x,t)), \qquad f^w(\gamma,x,t) := \partial_w f(\gamma u_\ast(x,t),\gamma w_\ast(x,t)).
\end{align*}
Since $U_\ast$ is bounded on $[0,T) \times \R$ and it holds $f(0,0) = 0$, it follows by the mean value theorem that the functions $F$ and $G$ are bounded on $[0,T) \times \R$ too. In addition, $F$ and $G$ are $\frac{\alpha}{2}$-H\"older continuous in $t$ and $\alpha$-H\"older continuous in $x$, since the same holds for $U_\ast$. 
Thus, by~\cite[Proposition 11.3]{ZUH} the Green's function $G(x,y,t,s)$ associated to~\eqref{linsys} is continuous, and differentiable with respect to $x$. Moreover, it enjoys the estimate
\begin{align} \left\|\partial_x^j G(x,y,t,s)\right\| \leq Ct^{-\tfrac{j+1}{2}} \re^{-\tfrac{(x-y)^2}{M(t-s)}}, \qquad x,y \in \R, 0 < s \leq t < T, j = 0,1, \label{Green}\end{align}
for some $x$-, $y$-, $s$- and $t$-independent constants $C,M > 1$. Let $T_0 \in (0,T)$. The Green's function estimate~\eqref{Green} and the fact that $U_\ast$ is bounded on $[0,T)$ imply that the solution
\begin{align*} U_\ast(x,t) = \int_\R G(x,y,t,T_0) U_\ast(y,T_0)\de y, \end{align*}
can be extended from $\R \times [T_0,T)$ to $\R \times [T_0,T]$ such that $U_\ast(\cdot,T) \in C_b^1(\R,\R^n)$. In particular, $U_\ast(\cdot,T)$ lies in $C_b^{0,\alpha}(\R,\R^{4n})$ and can therefore be extended by~\cite[Corollary 11.4]{ZUH} to a solution $U_\ast(t)$ in $C_b^{0,\alpha}(\R,\R^{4n})$ on some interval $[0,T+\tau)$ with $\tau > 0$, which contradicts the maximality of $T$. Thus, the blow-up~\eqref{blowuplinfty} must hold if $T < \infty$.
\end{proof}

\begin{remark} \label{remspaces}
{\upshape
We have established local existence of classical solutions to~\eqref{GRD} with initial data in the range $X_\alpha := \{\F^{-1}(v) : v \in W^{1,1}_{1+\alpha}(\R,\C^n), v(-k) = \overline{v(k)} \text{ for all } k \in \R\} \subset L^2(\R,\R^n)$ of the inverse Fourier transform restricted to all $v \in \smash{W^{1,1}_{1+\alpha}(\R,\C^n)}$ satisfying the reality condition. We note that the more `natural' algebraically weighted Sobolev space
\begin{align*}H^2_1(\R,\R^n) := \left\{u \in H^2(\R,\R^n) : \|u\|_{\smash{H^2_1}} < \infty\right\}, \end{align*}
equipped with the norm $\|u\|_{\smash{H^2_1}} = \left\|\varrho u\right\|_{H^2}$, where $\varrho \colon \R \to \R$ denotes the smooth algebraic weight $\varrho(x) = (1+x^2)^{1/2}$, is continuously embedded in the space $X_\alpha$. Of course, initial data in $H^2_1(\R,\R^n)$ are in general more regular and stronger localized than initial data in $X_\alpha$. This can be seen by looking at weighted fractional Sobolev spaces. The standard fractional Sobolev spaces $W^{s,p}(\R,\R^n)$ for $s \in \R_{> 0} \setminus \Nn$ and $p \in (1,\infty)$ are defined by
\begin{align*}
W^{s,p}(\R,\R^n) = \left\{u \in W^{\lfloor s \rfloor,p}(\R,\R^n) : [u]_{s - \lfloor s \rfloor,p} < \infty\right\}, \quad [u]_{\theta,p} := \left(\int_\R \int_\R \frac{\left\|D^{\lfloor s \rfloor}\left(u(x) - u(y)\right)\right\|^p}{|x-y|^{\theta p+1}} \de x \de y\right)^{\frac{1}{p}},
\end{align*}
and are equipped with the Slobodeckij norm $\|u\|_{W^{s,p}} = \|u\|_{W^{\lfloor s \rfloor,p}} + [u]_{s-\lfloor s \rfloor,p}$ or the equivalent Bessel norm $\|u\|_{H^{s,p}} = \left\|\left(1-\Delta\right)^{\frac{s}{2}}u\right\|_p$, where the fractional operator $\left(1-\Delta\right)^{\frac{s}{2}}$ corresponds to multiplication with $\varrho^s$ in Fourier space. We introduce the weighted fractional Sobolev spaces
\begin{align*} W_1^{s,p}(\R,\R^n) = \left\{u \in W^{s,p}(\R,\R^n) : \left\|\varrho^p u\right\|_{W^{s,p}} < \infty\right\}, \end{align*}
for $s \in \R_{>0} \setminus \Nn$ and $p \in (1,\infty)$. We equip $W_1^{s,p}(\R,\R^n)$ with the norm $\|u\|_{\smash{W_1^{s,p}}} = \|\varrho^p u\|_{\smash{W^{s,p}}}$. One readily observes via the H\"older and Babenko-Beckner inequalities that all spaces $\smash{W_1^{1+\alpha,p}}(\R,\R^n)$ with $\alpha p > 1$ are continuously embedded in $X_\alpha$. Thus, intuitively speaking, for initial data to lie in $X_\alpha$ for some $\alpha > 0$, it is enough to be more regular than one time differentiable and exhibit stronger decay than $1/(1+|x|)$ as $x \to \pm \infty$.}
\end{remark}

\section{Global analysis: proof of Theorem~\ref{mainresult1}} \label{sec:proofMIX}

In this proof, $C\geq1$ denotes a constant, which is independent of $\delta$ and $t$ and that will be taken larger if necessary.

\subsection{Plan of proof} \label{sec:planM}

Let $v_0 \in W_{1+\alpha}^{1,1}(\R,\C^n)$ with $v_0(-k) = \overline{v_0(k)}$ for each $k \in \R$. By Proposition~\ref{proplocal} there exists $T > 0$ such that we have a unique local solution $u \in C^{1,\frac{\alpha}{2}}\left((0,T),C_b^{3,\alpha}(\R,\R^n)\right)$ to~\eqref{GRD} with initial condition $u(0) = \F^{-1}(v_0)$. In addition, the function $v \colon [0,T) \to W^{1,1}_1(\R,\C^n)$ given by $v(t) = \F(u(t))$ is continuous and $T > 0$ is maximal in the sense that, if it holds $T < \infty$, then we have
\begin{align}\limsup_{t \uparrow T} \left\|v\right\|_{W^{1,1}_1} = \infty.\label{blowuplinfty2}\end{align}

To exploit oscillations arising in Fourier space due to differences in velocities we switch to an appropriate co-moving frame in each component. Thus, we define the new coordinate
\begin{align} w(k,t) := \Phi(k,t) v(k,t) = \Phi(k,t)\left(\F u\right)(k,t), \qquad \Phi(k,t) := \re^{-\ri \Ce kt}, \label{defw}\end{align}
for $k \in \R$ and $t \in [0,T)$. We aim to establish global control on the $W^{1,1}_1$-norm of $w(t)$. Thus, we introduce the temporal weight function $\eta \colon [0,T) \to \R$ given by
\begin{align*} \eta(t) = \sup_{s \in [0,t]} &\left[\sqrt{1+s}\|w(s)\|_1 + \frac{\sqrt{1+s}}{\ln(2+s)}\left\||\cdot| \partial_k w(s)\right\|_1 + (1+s)\left\||\cdot| w(s)\right\|_1 \right.\\
&\left.\qquad \qquad \qquad \qquad \qquad \phantom{\frac{\sqrt{1+s}}{\ln(2+s)}} + \left\|\partial_k w(s)\right\|_1 + (1+s)^{\frac{3}{4}} \left\||\cdot| w(s)\right\|_2 \right].
\end{align*}
We show in~\S\ref{sec:eta} that $\eta$ is well-defined and continuous and, in case $T < \infty$, it holds
\begin{align} \limsup_{t \uparrow T} \eta(t) = \infty. \label{contr}\end{align}
We remark that, although $W^{1,1}_1(\R,\C^n)$ is continuously embedded in $L^2_1(\R,\C^n)$, we need to include the $\left\||\cdot|w(s)\right\|_2$-term in $\eta(t)$ in order to obtain the desired estimates; we refer to Remark~\ref{L2boundnec} for more details.

Our plan is to prove via a continuous induction argument that $\eta$ is bounded and, consequently,~\eqref{contr} yields $T = \infty$. More specifically, we prove in~\S\ref{sec:key} that, if we have $\|v_0\|_{\smash{W^{1,1}_1}} \leq \delta$ and $t \in [0,T)$ is such that $\eta(t) \leq r_0$ (where $r_0 > 0$ is the constant given by the hypotheses of Theorem~\ref{mainresult1}), then $\eta(t)$ satisfies an inequality of the form
\begin{align} \eta(t) \leq C\left(\delta + \eta(t)^2\right). \label{etaest}\end{align}
Since $\eta$ must be continuous as long as it is bounded by~\eqref{contr}, we can apply continuous induction using~\eqref{etaest}. Thus, taking $\delta \leq \min\{\frac{1}{4C^2},\frac{r_0}{2C}\}$, it follows $\eta(t) \leq 2C\delta \leq r_0$ for \emph{all} $t \geq 0$, which proves global existence. Finally, we take $\delta = \min\{\frac{\epsilon}{2C},\frac{1}{4C^2},\frac{r_0}{2C}\}$, so that it holds
\begin{align}\eta(t) \leq 2C\delta \leq \epsilon, \label{etaest2} \end{align}
for $t \geq 0$. Since~\eqref{defw} implies
\begin{align*} \F^{-1}(w_i(t))(x) = u_i(x - c_i t,t), \qquad \F^{-1}(\partial_k w_i(t))(x) = -\ri x u_i(x-c_it,t), \qquad t \geq 0, x \in \R,\end{align*}
for $i = 1,\ldots,n$, the estimates~\eqref{temporaldec} and~\eqref{pointdec} follow from~\eqref{etaest2} and the fact that the Fourier transform maps $W^{1,1}_1(\R,\C^n)$ continuously into $W^{1,\infty}_1(\R,\C^n)$ with norm $\leq \frac{1}{2\pi}$.

Thus, all that remains is to show that $\eta$ is well-defined and continuous, that $T < \infty$ implies~\eqref{contr} and that $\eta$ satisfies the key estimate~\eqref{etaest}. We prove the first two assertions in~\S\ref{sec:eta}. The key estimate, which is the core of our global analysis, is shown in~\S\ref{sec:key}.

\subsection{Continuity and blow-up property of weight function} \label{sec:eta}

Since we have
\begin{align*}
\left\|w(t)\right\|_{\smash{W^{1,1}_1}} = \left\|(1+|\cdot|) v(t)\right\|_1 + \left\|(1+|\cdot|)\left(\partial_k v(t) - \ri \Ce t v(t)\right)\right\|_1
&\leq C(1+t)\|v(t)\|_{\smash{W^{1,1}_1}},
\end{align*}
for $t \in [0,T)$ and since $W^{1,1}_1(\R,\C^n)$ is continuously embedded into $L_1^2(\R,\C^n)$, the function $\eta$ is well-defined.

Next, we prove $\eta$ is continuous. Since $W^{1,1}_1(\R,\C^n)$ is continuously embedded into $L^2_1(\R,\C^n)$, it holds
\begin{align*}
\left|\left\||\cdot|^j w(t)\right\|_p - \left\||\cdot|^j w(s)\right\|_p\right| &= \left|\left\||\cdot|^j v(t)\right\|_p - \left\||\cdot|^j v(s)\right\|_p\right|
\leq C\|v(t) - v(s)\|_{{W^{1,1}_1}},
\end{align*}
for $s,t \in [0,T)$, $p = 1,2$ and $j = 0,1$. Hence, because $v \colon [0,T) \to W^{1,1}_1(\R,\C^n)$ is continuous, also $t \mapsto \left\||\cdot|^j w(t)\right\|_p$ is continuous on $[0,T)$ for $j = 0,1$ and $p = 1,2$. Second, we establish
\begin{align*}
\begin{split}
\left|\left\||\cdot|^j \partial_k w(t)\right\|_1 - \left\||\cdot|^j \partial_k w(s)\right\|_1\right|
&= \left|\left\||\cdot|^j \left(\partial_k v(t) - \ri \Ce t v(t)\right)\right\|_1 - \left\||\cdot|^j \left(\partial_k v(s) - \ri \Ce sv(s)\right)\right\|_1\right|\\
&\leq C\left(\left\||\cdot|^j \partial_k \left(v(t) - v(s)\right)\right\|_1 + |t-s|\left\||\cdot|^j \left(v(t) - v(s)\right)\right\|_1\right)\\
&\leq C\left(1 +|t-s|\right)\|v(t) - v(s)\|_{\smash{W^{1,1}_1}},
\end{split}
\end{align*}
for $s,t \in [0,T)$ and $j = 0,1$. So, since $v \colon [0,T) \to W^{1,1}_1(\R,\C^n)$ is continuous, also $t \mapsto \left\||\cdot|^j \partial_k w(t)\right\|_1$ is continuous on $[0,T)$ for $j = 0,1$. Therefore, $\eta$ must be continuous.

Finally, the fact that $T < \infty$ implies~\eqref{contr} follows from~\eqref{blowuplinfty2} and the estimate
\begin{align*}
\left\|v(t)\right\|_{\smash{W^{1,1}_1}} = \left\|(1+|\cdot|) w(t)\right\|_1 + \left\|(1+|\cdot|)\left(\partial_k w(t) + \ri \Ce t w(t)\right)\right\|_1
&\leq C\left(1+\sqrt{t}\right)\eta(t),
\end{align*}
for $t \in [0,T)$.

\subsection{Establishing the key estimate} \label{sec:key}

We integrate~\eqref{GRD}, apply the Fourier transform and multiply with $\Phi(k,t)$ to arrive at the Duhamel formulation:
\begin{align} w(k,t) = \re^{-k^2 \D t} v_0(k) + \int_0^{t} \re^{-k^2 \D(t-s)} \widetilde{\N}(k,s) \de s, \qquad k \in \R, t \in [0,T), \label{Duhamel3}\end{align}
with
\begin{align*} \widetilde{\N}(k,s) &:= \Phi(k,s) \F\left[f\left(u(s),\partial_x u(s)\right)\right](k),
\end{align*}
cf.~\eqref{Duhamel1} and~\eqref{defw}. It follows from~\eqref{Duhamel3} that $w(k,t)$ is pointwise differentiable with respect to $t$ and satisfies the differential equation
\begin{align}
\partial_t w(k,t) = -k^2\D w(k,t) + \widetilde{\N}(k,t), \qquad t \in [0,T), \, k \in \R. \label{wequation}
\end{align}

To isolate the marginal nonlinear terms we expand the nonlinearity $f$ in~\eqref{GRD}. Thus, by~\eqref{nonlinearbounds1}, the $i$-th component $f_i \in C^4(\R^n \times \R^n,\R)$ of $f$ can be expanded as
\begin{align*}
f_i(a,b) = \sum_{j = 1}^n \sum_{l = 1}^n \mu_{ijl} a_j b_l + \sum_{j = 1}^n \sum_{l = 1}^n \sum_{\begin{smallmatrix} m \in \{1,\ldots,n\},\\ m \neq l\end{smallmatrix}} \! \nu_{ijlm} a_ja_la_m + g_i(a,b),
\end{align*}
with coefficients $\mu_{ijl}, \nu_{ijlm} \in \R$ and remainder $g_i \in C^0(\R^n \times \R^n,\R)$ satisfying
\begin{align}
\|g_i(a,b)\| \leq C\left(\|a\|^4 + \|b\|^2\right), \label{irrelbound}
\end{align}
for $i = 1,\ldots,n$ and $a, b \in \R^n$ with $\|a\|,\|b\| \leq r_0$. Hence, the Duhamel formulation of the $i$-th component of $w$ reads
\begin{align}
w_i(k,t) &= I_i(k,t) + R_i(k,t) + \sum_{j = 1}^n \sum_{l = 1}^n M_{ijl}(k,t) + \sum_{j = 1}^n \sum_{l = 1}^n \sum_{\begin{smallmatrix} m \in \{1,\ldots,n\},\\ m \neq l\end{smallmatrix}} N_{ijlm}(k,t),  \label{Duhamel}
\end{align}
for $k \in \R$ and $t \in [0,T)$, with
\begin{align*}
\check{\N}_i(k,t) &:= \Phi(k,t) \F\left[g_i\left(u(t),\partial_x u(t)\right)\right](k),\\
I_i(k,t) := \re^{-d_i k^2 t} v_{0,i}(k), &\qquad \qquad \qquad \qquad R_i(k,t) := \int_0^t \re^{-d_i k^2(t-s)} \check{\N}_i(k,s) \de s,
\end{align*}
and
\begin{align*}
M_{ijl}(k,t) &:= \frac{\mu_{ijl}}{2\pi} \int_0^t \int_\R \re^{-d_i k^2(t-s) + (c_j - c_i)\ri k s + (c_l - c_j)\ri \xi s} w_j(k-\xi,s) \, \ri \xi  w_l(\xi,s) \de \xi \de s,\\
N_{ijlm}(k,t) &:= \frac{\nu_{ijlm}}{4\pi^2} \int_0^t \int_\R \int_\R \re^{-d_i k^2(t-s) + (c_j - c_i)\ri k s + (c_l - c_j)\ri \xi s + (c_m - c_l)\ri \zeta s} \\
&\qquad \qquad \qquad \qquad \qquad \qquad \times w_j(k-\xi,s)w_l(\xi-\zeta,s)w_m(\zeta,s) \de \zeta \de \xi \de s,
\end{align*}
for $i,j,l,m \in \{1,\ldots,n\}$. Our plan is to prove the key inequality~\eqref{etaest}, provided $t \in [0,T)$ is such that $\eta(t) \leq r_0$, by estimating the linear term $I_i(\cdot,t)$ and nonlinear terms $R_i(\cdot,t), M_{ijl}(\cdot,t)$ and $N_{ijlm}(\cdot,t)$ in~\eqref{Duhamel} one by one in $W^{1,1}_1(\R,\C^n)$ for $i,j,l,m \in \{1,\ldots,n\}$ with $m \neq l$.

\subsubsection{Embedding in \texorpdfstring{$L^2(\R,\C^n)$}{ } and \texorpdfstring{$L^\infty_1(\R,\C^n)$}{ }}

Take $t \in [0,T)$. To bound those integrals in~\eqref{Duhamel} corresponding to the nonlinear terms, we need control over the $L^2_1$- and $L^\infty_1$-norm of $w(s)$ for $s \in [0,t]$. Thus, take $t \in [0,T)$. Since $W^{1,1}(\R,\C^n)$ is continuously embedded in $L^\infty(\R,\C^n)$, we have, by definition of the weight $\eta$, the following bounds:
\begin{align}
\begin{split}
\|w(s)\|_\infty &\leq C\| w(s)\|_{W^{1,1}} \leq C\left(\|w(s)\|_1 + \|\partial_k w(s)\|_1\right) \leq C\eta(t), \\
\||\cdot|w(s)\|_\infty &\leq C\||\cdot| w(s)\|_{W^{1,1}} \leq C\left(\||\cdot| w(s)\|_1 + \||\cdot| \partial_k w(s)\|_1 + \|w(s)\|_1\right) \leq C\frac{\eta(t) \ln(2+s)}{\sqrt{1+s}}.
\end{split}
\label{linftyemb}
\end{align}
for $s \in [0,t]$. Hence, interpolation yields
\begin{align}
\begin{split}
\|w(s)\|_2 &\leq C\sqrt{\|w(s)\|_1 \|w(s)\|_\infty} \leq C\frac{\eta(t)}{(1+s)^{\frac{1}{4}}}, \qquad s \in [0,t].
\end{split}
\label{l2emb}
\end{align}

\begin{remark} \label{L2boundnec}
{\upshape
We expect that the bound
\begin{align}
\||\cdot| w(s)\|_2 &\leq C\sqrt{\||\cdot| w(s)\|_1 \, \||\cdot|w(s)\|_\infty} \leq C\frac{\eta(t) \sqrt{\ln(2+s)}}{(1+s)^{\frac{3}{4}}}, \qquad s \in [0,t], \label{l2bad}
\end{align}
obtained through interpolation, is not strong enough to close the nonlinear iteration scheme. Indeed,~\eqref{l2bad} would introduce a logarithm in~\eqref{irrb2}, which would lead to a $\smash{\!\sqrt{\ln(2+s)}}$-factor in the bound on $\|\partial_k w(s)\|_1$ via~\eqref{irrbad} and, thus, on $\|w(s)\|_\infty$ in~\eqref{linftyemb}, which we expect cannot be accommodated for. This is the reason why we include $\||\cdot| w(s)\|_2$ in our temporal weight function $\eta(t)$.
}\end{remark}

\subsubsection{Linear estimates} \label{sec:lin}

Let $i \in \{1,\ldots,n\}$. First, since $W^{1,1}_1(\R,\C^n)$ is continuously embedded in $L^\infty(\R,\C^n)$, we have for $j = 0,1$ the estimate
\begin{align*} \left\||\cdot|^j I_i(t)\right\|_1 &= \int_\R \left|k^j \re^{-k^2 d_i t} v_{0,i}(k)\right| \de k
\begin{cases} \leq C\left\||\cdot|^j v_{0,i}\right\|_1 \leq C \|v_0\|_{\smash{W^{1,1}_1}} \leq C\delta, & t \in [0,T),\\
\leq Ct^{-\frac{1+j}{2}} \left\|v_{0,i}\right\|_\infty \leq C t^{-\frac{1+j}{2}} \|v_0\|_{\smash{W^{1,1}_1}} \leq C\delta t^{-\frac{1+j}{2}},& t \in (0,T).
\end{cases}
\end{align*}
Moreover, since $W^{1,1}_1(\R,\C^n)$ is continuously embedded in $L^\infty(\R,\C^n)$ and in $L^2_1(\R,\C^n)$, it holds
\begin{align*} \left\||\cdot| I_i(t)\right\|_2 &= \left(\int_\R \left|k \re^{-k^2 d_i t} v_{0,i}(k)\right|^2 \de k\right)^{\frac{1}{2}}
\begin{cases} \leq C\left\||\cdot| v_{0,i}\right\|_2 \leq C \|v_0\|_{\smash{W^{1,1}_1}} \leq C\delta, & t \in [0,T),\\
\leq Ct^{-\frac{3}{4}} \left\|v_{0,i}\right\|_\infty \leq C t^{-\frac{3}{4}} \|v_0\|_{\smash{W^{1,1}_1}} \leq C\delta t^{-\frac{3}{4}},& t \in (0,T).
\end{cases}
\end{align*}
Next, we establish
\begin{align*} \left\||\cdot| \partial_k I_i(t)\right\|_1 &\leq C\left(\int_\R \left|k^2 t \re^{-k^2 d_i t} v_{0,i}(k)\right| \de k +  \int_\R \left|k \re^{-k^2 d_i t} \partial_k v_{0,i}(k)\right| \de k\right)\\
&\leq \begin{cases} C\left(\|v_{0,i}\|_1 + \left\||\cdot| \partial_k v_{0,i}\right\|_1\right) \leq  C \|v_0\|_{\smash{W^{1,1}_1}} \leq C\delta, & t \in [0,T),\\
\frac{C}{\sqrt{t}}\left(\|v_{0,i}\|_\infty + \left\|\partial_k v_{0,i}\right\|_1\right) \leq  \frac{C}{\sqrt{t}} \|v_0\|_{\smash{W^{1,1}_1}} \leq C\frac{\delta}{\sqrt{t}}, & t \in (0,T).
\end{cases}
\end{align*}
Finally, it holds
\begin{align*} \left\|\partial_k I_i(t)\right\|_1 &\leq C\left(\int_\R \left|k t \re^{-k^2 d_i t} v_{0,i}(k)\right| \de k +  \int_\R \left|\re^{-k^2 d_i t} \partial_k v_{0,i}(k)\right| \de k\right)\\
&\leq C\left(\|v_{0,i}\|_\infty + \left\|\partial_k v_{0,i}\right\|_1\right) \leq  C \|v_0\|_{\smash{W^{1,1}_1}} \leq C\delta.
\end{align*}
for $t \in [0,T)$. All in all, we have established the linear estimates
\begin{align}
\left\||\cdot|^j \partial_k^m I_i(t)\right\|_1 \leq C\delta(1+t)^{-\frac{1+j-m}{2}}, \qquad \left\||\cdot| I_i(t)\right\|_2 \leq C\delta(1+t)^{-\frac{3}{4}}, \label{linest}
\end{align}
for $t \in [0,T)$, $j = 0,1$, $m = 0,1$ and $i \in \{1,\ldots,n\}$.

\subsubsection{Estimates on irrelevant nonlinear terms}

Let $i \in \{1,\ldots,n\}$ and let $t \in [0,T)$ be such that $\eta(t) \leq r_0$. For $s \in [0,t]$ and $j = 0,1$, we have by~\eqref{defFourier},~\eqref{defw} and~\eqref{l2emb} the estimate
\begin{align}
\left\|\partial_x^j u(s)\right\|_\infty &\leq \frac{1}{2\pi} \left\||\cdot|^j v(s)\right\|_1 \leq \left\||\cdot|^j w(s)\right\|_1 \leq \frac{\eta(t)}{(1+s)^{\frac{1+j}{2}}} \leq r_0, \label{A1}
\end{align}
and
\begin{align}
\begin{split}
\left\|\partial_x^j u(s)\right\|_2 &\leq C\left\||\cdot|^j v(s)\right\|_2 = C\left\||\cdot|^j w(s)\right\|_2 \leq C\frac{\eta(t)}{(1+s)^{\frac{1+2j}{4}}},\\
\left\||\cdot| \partial_x^j u(s)\right\|_\infty &\leq C\left\|\partial_k\left((\cdot)^j v(s)\right)\right\|_1 \leq C\left(\left\||\cdot|^j \partial_k w(s)\right\|_1 + \left\|w(s)\right\|_1 + s\left\||\cdot|^j w(s)\right\|_1\right)\\ & \leq C\eta(t)(1+s)^{\frac{1-j}{2}}.
\end{split} \label{A2}
\end{align}
Thus,~\eqref{irrelbound},~\eqref{A1} and~\eqref{A2} yield
\begin{align}
\begin{split}
\left\|\check{\N}_i(\cdot,s)\right\|_\infty &\leq C\left\|g_i\left(u(s),\partial_x u(s)\right)\right\|_1 \leq C\left(\|u(s)\|_\infty^2\|u(s)\|_2^2 + \|\partial_x u(s)\|_2^2\right)
\leq C\frac{\eta(t)^2}{(1+s)^{\frac{3}{2}}},\\
\left\|\check{\N}_i(\cdot,s)\right\|_2 &\leq C\left\|g_i\left(u(s),\partial_x u(s)\right)\right\|_2 \leq C\left(\|u(s)\|_\infty^3\|u(s)\|_2 + \|\partial_x u(s)\|_\infty\|\partial_x u(s)\|_2\right) \leq C\frac{\eta(t)^2}{(1+s)^{\frac{7}{4}}},
\end{split}\label{irrb1}
\end{align}
and
\begin{align}
\begin{split}
\left\|\partial_k \check{\N}_i(\cdot,s)\right\|_2 &\leq C\left(s\left\|g_i\left(u(s),\partial_x u(s)\right)\right\|_2 + \left\||\cdot| g_i\left(u(s),\partial_x u(s)\right)\right\|_2\right)\\
&\leq C\left(s\left\|g_i\left(u(s),\partial_x u(s)\right)\right\|_2 + \||\cdot| u(s)\|_\infty \|u(s)\|_\infty^2\|u(s)\|_2 + \||\cdot| \partial_x u(s)\|_\infty\|\partial_x u(s)\|_2\right)\\ &\leq C\frac{\eta(t)^2}{(1+s)^{\frac{3}{4}}},
\end{split} \label{irrb2}
\end{align}
for $s \in [0,t]$. As in estimates~\eqref{irrill1} and~\eqref{irrill2}, we estimate for $j = 0,1$ using the first equation in~\eqref{irrb1}:
\begin{align*} \left\||\cdot| R_i(t)\right\|_2 &\leq C \int_0^t \left(\int_\R \left|k \re^{-d_i k^2(t-s)} \check{\N}_i(k,s)\right|^2 \de k\right)^{\frac{1}{2}} \de s \leq C \int_0^t \frac{\eta(t)^2}{(t-s)^{\frac{3}{4}}(1+s)^{\frac{3}{2}}} \de s \leq C\frac{\eta(t)^2}{(1+t)^{\frac{3}{4}}},\end{align*}
and
\begin{align*} \left\||\cdot|^j R_i(t)\right\|_1 &\leq C \int_0^t \int_\R \left|k^j \re^{-d_i k^2(t-s)} \check{\N}_i(k,s)\right| \de k \de s\\
&\leq C\eta(t)^2 \left(\int_0^{\frac{t}{2}} \frac{1}{(t-s)^{\frac{1+j}{2}}(1+s)^{\frac{3}{2}}} \de s + \int_{\frac{t}{2}}^t \frac{1}{(t-s)^{\frac{1+2j}{4}}(1+s)^{\frac{7}{4}}} \de s \right) \leq C\frac{\eta(t)^2}{(1+t)^{\frac{1+j}{2}}}.
\end{align*}
Similarly,~\eqref{irrb1} and~\eqref{irrb2} yield
\begin{align}
\label{irrbad}
\begin{split}
\left\||\cdot|^j \partial_k R_i(t)\right\|_1 &\leq C\left( \int_0^t\! \int_\R \left|k^{j+1} (t-s) \re^{-d_i k^2(t-s)} \check{\N}_i(k,s)\right| \de k \de s\right. \\
&\qquad \qquad \qquad \qquad \left. + \, \int_0^t \int_\R \left|k^j \re^{-d_i k^2(t-s)} \partial_k \check{\N}_i(k,s)\right| \de k \de s\right) \\
&\leq C\eta(t)^2 \left(\int_0^t \frac{1}{(t-s)^{\frac{j}{2}}(1+s)^{\frac{3}{2}}} \de s + \int_0^t \frac{1}{(t-s)^{\frac{1+2j}{4}}(1+s)^{\frac{3}{4}}} \de s\right)
\leq C\frac{\eta(t)^2}{(1+t)^{\frac{j}{2}}},
\end{split}
\end{align}
for $j = 0,1$. All in all, we have established the nonlinear estimates
\begin{align}
\left\||\cdot|^j \partial_k^m R_i(t)\right\|_1 \leq C\frac{\eta(t)^2}{(1+t)^{\frac{1+j-m}{2}}}, \qquad \left\||\cdot| R_i(t)\right\|_2 \leq C\frac{\eta(t)^2}{(1+t)^{\frac{3}{4}}}, \label{nonlinestR}
\end{align}
for $t \in [0,T)$, $j = 0,1$, $m = 0,1$ and $i \in \{1,\ldots,n\}$.

\subsubsection{Short-time bounds on marginal terms with derivatives}

Let $i,j,l \in \{1,\ldots,n\}$ and let $t \in [0,T)$. As in~\S\ref{sec:illfreq}, we split our estimates on $M_{ijl}(t)$ in short- and large-time estimates. In this subsection, we establish short-time bounds on $M_{ijl}(t)$. Large-time estimates are then obtained in~\S\ref{sec:Burgers} and~\S\ref{bound:M}. Thus, for $t \leq 2$, $a = 0,1$ and $b = 0,1$, we establish
\begin{align}
\begin{split}
&\left\||\cdot|^a \partial_k^b M_{ijl}(t)\right\|_1 \leq C\left(\int_0^t \int_\R \int_\R \left|k^a \re^{-d_i k^2(t-s)} \partial_k^b w_j(k-\xi,s) \, \xi w_l(\xi,s)\right| \de \xi \de k \de s\right.\\
&\qquad \quad \left. + \, \int_0^t \int_\R \int_\R \left|k^a\left(d_i|k|(t-s) + |c_j-c_i|s\right) \re^{-d_i k^2(t-s)} w_j(k-\xi,s) \, \xi w_l(\xi,s)\right| \de \xi \de k \de s\right)\\
&\quad \leq C\eta(t)^2 \int_0^t \frac{1}{(t-s)^{\frac{a}{2}}} \de s \leq C\frac{\eta(t)^2}{(1+t)^{\frac{1+a-b}{2}}},
\end{split} \label{shorttime}
\end{align}
and
\begin{align}
\begin{split}
\left\||\cdot| M_{ijl}(t)\right\|_2 &\leq C \int_0^t \left(\int_\R \left(\int_\R \left|k \re^{-d_i k^2(t-s)} w_j(k-\xi,s) \, \xi w_l(\xi,s) \right| \de \xi\right)^2 \de k \right)^{\frac{1}{2}} \de s\\ & \leq C\eta(t)^2 \int_0^t \frac{1}{(t-s)^{\frac{3}{4}}} \de s \leq C\frac{\eta(t)^2}{(1+t)^{\frac{3}{4}}}.
\end{split} \label{shorttime2}
\end{align}

\subsubsection{Estimates on Burgers'-type terms} \label{sec:Burgers}

Burgers'-type terms yield integrals of the form $M_{ijj}(t)$ in the Duhamel formulation~\eqref{Duhamel}, which can be rewritten using
\begin{align}
\begin{split}
&\int_0^t \int_\R \re^{-d_i k^2(t-s) + (c_j - c_i)\ri k s} w_j(k-\xi,s) \, \ri \xi  w_j(\xi,s) \de \xi \de s= \frac{\ri k}{2} \int_0^t \re^{-d_i k^2(t-s) + (c_j - c_i)\ri k s}  w_j^{*2}(k,s)  \de s,
\end{split}
\label{BurgINT}
\end{align}
with $k \in \R$ and $t \in [0,T)$. In case $i \neq j$, we have $c_i \neq c_j$ due to differences in velocities and the exponential in~\eqref{BurgINT} is oscillatory in $s$. We exploit these oscillations by integrating by parts in the temporal variable $s$. We emphasize that such an integration could introduce singularities at $k = 0$, which are however cancelled by the factor $k$ in front of the integral in~\eqref{BurgINT}.

Thus, let $i,j \in \{1,\ldots,n\}$ and let $t \in [2,T)$ be such that $\eta(t) \leq r_0$. We use Young's convolution inequality,~\eqref{l2emb} and~\eqref{BurgINT} to bound
\begin{align*}
\left\||\cdot| M_{ijj}(t)\right\|_2 &\leq C \left(\int_0^{\frac{t}{2}} \left(\int_\R \left|k^2 \re^{-d_i k^2(t-s)} w_j^{*2}(k,s)\right|^2 \de k\right)^{\frac{1}{2}} \de s\right.\\
&\qquad \qquad \qquad \left. + \, \int_{\frac{t}{2}}^t \left(\int_\R \left|k \re^{-d_i k^2(t-s)} \int_\R w_j(k-\xi,s) \, \xi w_j(\xi,s) \de \xi \right| \de k\right)^{\frac{1}{2}} \de s\right)\\
&\leq C \eta(t)^2\left(\int_0^{\frac{t}{2}} \frac{1}{(t-s) (1+s)^{\frac{3}{4}}} \de s + \int_{\frac{t}{2}}^t \frac{1}{(t-s)^{\frac{3}{4}} (1+s)} \de s\right) \leq C\frac{\eta(t)^2}{(1+t)^{\frac{3}{4}}}.
\end{align*}
Similarly, for $a = 0,1$ we use~\eqref{linftyemb} and~\eqref{BurgINT} to estimate
\begin{align*}
\left\||\cdot|^a M_{ijj}(t)\right\|_1 &\leq C \left(\int_0^{\frac{t}{2}} \int_\R \left|k^{a+1} \re^{-d_i k^2(t-s)} w_j^{*2}(k,s)\right| \de k \de s\right.\\
&\qquad \qquad \qquad \left. + \, \int_{\frac{t}{2}}^t \int_\R \int_\R \left|k^a \re^{-d_i k^2(t-s)} w_j(k-\xi,s) \, \xi w_j(\xi,s) \right| \de \xi\de k \de s\right)\\
&\leq C \eta(t)^2 \left(\int_0^{\frac{t}{2}} \frac{1}{(t-s)^{1+\frac{a}{2}} \sqrt{1+s} } \de s + \int_{\frac{t}{2}}^t \frac{1}{(t-s)^{\frac{a}{2}} (1+s)^{\frac{3}{2}}} \de s\right) \leq C\frac{\eta(t)^2}{(1+t)^{\frac{1+a}{2}}}.
\end{align*}
The $k$-derivative of~\eqref{BurgINT} is the sum of the following three integrals
\begin{align*}
\I_{1,ij}(k,t) &:= \frac{\ri}{2} \int_0^t \left(-2d_i k^2(t-s) + 1\right) \re^{-d_i k^2(t-s) + (c_j - c_i)\ri k s} w_j^{*2}(k,s) \de s,\\
\I_{2,ij}(k,t) &:= \frac{\ri k}{2} \int_0^t \re^{-d_i k^2(t-s) + (c_j - c_i)\ri k s}  \partial_k \left(w_j^{*2}(k,s)\right) \de s,\\
\I_{3,ij}(k,t) &:= \frac{(c_i - c_j) k}{2} \int_0^t s \, \re^{-d_i k^2(t-s) + (c_j - c_i)\ri k s} w_j^{*2}(k,s) \de s,
\end{align*}
with $k \in \R$. In order to bound $\left\||\cdot|^a \partial_k M_{ijj}(t)\right\|_1$ for $a = 0,1$, we estimate these three integrals one by one. First, by~\eqref{l2emb} we have
\begin{align*}
\left\||\cdot|^a \I_{1,ij}(t)\right\|_1 &\leq C \int_0^t \int_\R \left|k^a \left(k^2(t-s) + 1\right)\re^{-d_i k^2(t-s)} w_j^{*2}(k,s)\right| \de k \de s\\
&\leq C \eta(t)^2\int_0^t \frac{1}{(t-s)^{\frac{1+2a}{4}} (1+s)^{\frac{3}{4}}} \de s \leq C\frac{\eta(t)^2}{(1+t)^{\frac{a}{2}}},
\end{align*}
for $a = 0,1$. For the second integral $\I_{2,ij}(t)$, we have, on the one hand, the estimate
\begin{align*}
\left\|\I_{2,ij}(t)\right\|_1 &\leq C\int_0^t \int_\R \left|k \re^{-d_i k^2(t-s)} \partial_k \left(w_j^{*2}(k,s)\right)\right| \de k \de s  \leq C \int_0^t \frac{\eta(t)^2}{\sqrt{t-s} \sqrt{1+s}} \de s \leq C \eta(t)^2.
\end{align*}
On the other hand, by~\eqref{l2emb} it holds
\begin{align*}
\left\||\cdot| \I_{2,ij}(t)\right\|_1 &\leq C\left(\int_0^t \int_\R \int_\R \left|k \re^{-d_i k^2(t-s)} (k-\xi) \partial_k w_j(k-\xi,s) w_j(\xi,s)\right| \de \xi \de k \de s\right. \\
&\qquad \qquad \left. + \, \int_0^t \int_\R \int_\R \left|k \re^{-d_i k^2(t-s)} \partial_k w_j(k-\xi,s) \xi \, w_j(\xi,s)\right| \de \xi \de k \de s\right)\\
&\leq C \eta(t)^2 \int_0^t \frac{\ln(2+s)}{(t-s)^{\frac{3}{4}} (1+s)^{\frac{3}{4}}} \de s \leq C\eta(t)^2 \frac{\ln(2+t)}{\sqrt{1+t}}.
\end{align*}
The last integral $\I_{3,ij}(t)$ vanishes if $i = j$. If $i \neq j$, then the exponential in $\I_{3,ij}(t)$ is oscillatory, since it holds $c_i \neq c_j$. Thus, assume $i \neq j$. Integration by parts yields
\begin{align*}
\I_{3,ij}(k,t) &= \frac{1}{2} \psi_{ij}(k) \left(\left[s\, \re^{-d_i k^2(t-s) + (c_j - c_i)\ri k s} w_j^{*2}(k,s)\right]_0^t - \int_0^t \re^{-d_i k^2(t-s) + (c_j - c_i)\ri k s} w_j^{*2}(k,s)\de s\right.\\
&\qquad\qquad \qquad \left. -\, \int_0^t s\,\re^{-d_i k^2(t-s) + (c_j - c_i)\ri k s} \partial_s \left(w_j^{*2}(k,s)\right) \de s\right),
\end{align*}
where we denote
\begin{align*}
\psi_{ij}(k) := \frac{c_i - c_j}{(c_j - c_i)\ri + d_i k}.
\end{align*}
Hence, because equation~\eqref{wequation} holds pointwise, $\I_{3,ij}(k,t)$ is the sum of the following five terms
\begin{align*}
\J_{1,ij}(k,t) &:= -\psi_{ij}(k) \int_0^t \int_\R s\,\re^{-d_i k^2(t-s) + (c_j - c_i)\ri k s} w_j(k-\xi,s) \widetilde{\N}_j(\xi,s) \de \xi \de s, \\
\J_{2,ij}(k,t) &:= d_j \,k\, \psi_{ij}(k) \int_0^t \int_\R s\,\re^{-d_i k^2(t-s) + (c_j - c_i)\ri k s} w_j(k-\xi,s) \, \xi w_j(\xi,s)\de \xi \de s, \\
\J_{3,ij}(k,t) &:= -\frac{1}{2} \psi_{ij}(k) \int_0^t \re^{-d_i k^2(t-s) + (c_j - c_i)\ri k s} w_j^{*2}(k,s) \de s,\\
\J_{4,ij}(k,t) &:= -d_j \psi_{ij}(k) \int_0^t \int_\R s\,\re^{-d_i k^2(t-s) + (c_j - c_i)\ri k s} (k-\xi) w_j(k-\xi,s) \, \xi w_j(\xi,s)\de \xi \de s, \\
\J_{5,ij}(k,t) &:= \frac{1}{2} \psi_{ij}(k)\, t\, \re^{(c_j - c_i) \ri k t} w_j^{*2}(k,t),
\end{align*}
for $k \in \R$, where we have
\begin{align*}
\widetilde{\N}_j(k,s) &:= \Phi(k,s) \F\left[f_j\left(u(s),\partial_x u(s)\right)\right](k), \qquad k \in \R, s \in [0,t].
\end{align*}
First, using~\eqref{nonlinearbounds1},~\eqref{A1} and~\eqref{A2} we establish
\begin{align*}
\begin{split}
\left\|\widetilde{\N}_j(\cdot,s)\right\|_2 &\leq C\left\|f_j\left(u(s),\partial_x u(s)\right)\right\|_2 \leq C\left(\|u(s)\|_\infty^2\|u(s)\|_2 + \|u(s)\|_\infty\|\partial_x u(s)\|_2 + \|\partial_x u(s)\|_\infty\|\partial_x u(s)\|_2\right)\\ & \leq C\frac{\eta(t)}{(1+s)^{\frac{5}{4}}},
\end{split}
\end{align*}
for $s \in [0,t]$. Hence, since $\psi_{ij}$ is bounded on $\R$, we arrive for $a = 0,1$ at
\begin{align*}
\left\||\cdot|^a \J_{1,ij}(t)\right\|_1 &\leq C \int_0^t \int_\R \int_\R \left|k^a \re^{-d_i k^2(t-s)} s w_j(k-\xi,s) \widetilde{\N}_j(\xi,s)\right| \de \xi \de k \de s\\
&\leq C \eta(t)^2 \int_0^t \frac{1}{(t-s)^{\frac{1+2a}{4}}(1+s)^{\frac{3}{4}}} \de s \leq C\frac{\eta(t)^2}{(1+t)^{\frac{a}{2}}}.
\end{align*}
Second, since also $k \mapsto k\psi_{ij}(k)$ is bounded on $\R$, it holds
\begin{align}
\label{log1}
\begin{split}
&\left\||\cdot|^a \J_{2,ij}(t)\right\|_1 \leq C\left(\int_0^{t-1} \int_\R \int_\R \left|k^{1+a} \re^{-d_i k^2(t-s)} s w_j(k-\xi,s) \, \xi w_j(\xi,s)\right|\de \xi \de k \de s\right.\\
&\qquad \qquad \qquad \qquad\left. + \, \int_{t-1}^t \int_\R \int_\R \left|k^{a} \re^{-d_i k^2(t-s)} s w_j(k-\xi,s) \, \xi w_j(\xi,s)\right|\de \xi \de k \de s\right)\\
&\qquad \qquad\leq C\eta(t)^2\left(\int_0^{t-1} \frac{1}{(t-s)^{\frac{1+a}{2}} \sqrt{1+s}} + \int_{t-1}^t \frac{1}{(t-s)^{\frac{a}{2}} \sqrt{1+s}} \de s\right) \leq C\frac{\eta(t)^2\left(\ln(2+t)\right)^a}{(1+t)^{\frac{a}{2}}}
\end{split}
\end{align}
for $a = 0,1$. Third, using~\eqref{l2emb}, we estimate
\begin{align*}
\left\||\cdot|^a \J_{3,ij}(t)\right\|_1 &\leq C \int_0^t \int_\R \left|k^a \re^{-d_i k^2(t-s) + (c_j - c_i)\ri k s} w_j^{*2}(k,s)\right| \de k \de s\\
&\leq C\eta(t)^2 \int_0^t \frac{1}{(t-s)^{\frac{1+2a}{4}}(1+s)^{\frac{3}{4}}} \de s \leq C\frac{\eta(t)^2}{(1+t)^{\frac{a}{2}}},
\end{align*}
and
\begin{align*}
\left\||\cdot|^a \J_{4,ij}(t)\right\|_1 &\leq C \int_0^t \int_\R \int_\R \left|k^a \re^{-d_i k^2(t-s)} s (k-\xi) w_j(k-\xi,s) \, \xi w_j(\xi,s)\right| \de \xi \de k \de s \\
&\leq C\eta(t)^2 \int_0^t \frac{1}{(t-s)^{\frac{1+2a}{4}}(1+s)^{\frac{3}{4}}} \de s \leq C\frac{\eta(t)^2}{(1+t)^{\frac{a}{2}}},
\end{align*}
for $a = 0,1$. Finally, we obtain
\begin{align*}
\left\||\cdot|^a \J_{5,ij}(t)\right\|_1 \leq C \int_\R \left|k^a t w_j^{*2}(k,t)\right| \de k \leq C\frac{\eta(t)^2}{(1+t)^{\frac{a}{2}}},
\end{align*}
for $a = 0,1$. The estimates on $\J_{b,ij}(t)$ for $b = 1,\ldots,5$ yield
\begin{align*}
\left\||\cdot|^a \I_{3,ij}(t)\right\|_1 \leq C\frac{\eta(t)^2\left(\ln(2+t)\right)^a}{(1+t)^{\frac{a}{2}}},
\end{align*}
for $a = 0,1$ and $i \neq j$, whereas $\I_{3,ij}(t)$ vanishes for $i = j$. Thus, combining the latter with the estimates on $\I_{1,ij}(t)$ and $\I_{2,ij}(t)$, we arrive at
\begin{align*}
\left\||\cdot|^a \partial_k M_{ijj}(t)\right\|_1 \leq C\frac{\eta(t)^2\left(\ln(2+t)\right)^a}{(1+t)^{\frac{a}{2}}},
\end{align*}
for $a = 0,1$. Finally, combining the estimates on $M_{ijj}(t)$ with the short-time bounds~\eqref{shorttime} and~\eqref{shorttime2}, we establish
\begin{align}
\left\||\cdot|^a \partial_k^b M_{ijj}(t)\right\|_1 \leq C\frac{\eta(t)^2\left(\ln(2+t)\right)^{ab}}{(1+t)^{\frac{1+a-b}{2}}}, \qquad \left\||\cdot| M_{ijj}(t)\right\|_2 \leq C\frac{\eta(t)^2}{(1+t)^{\frac{3}{4}}}, \label{nonlinestM1}
\end{align}
for $t \in [0,T)$, $a = 0,1$, $b = 0,1$ and $i,j \in \{1,\ldots,n\}$.

\subsubsection{Estimates on marginal mixed-terms with derivatives} \label{bound:M}

All marginal mixed-terms with derivatives yield integrals of the form $M_{ijl}(t)$ with $j \neq l$ in the Duhamel formulation~\eqref{Duhamel}. Since we have $c_l \neq c_j$ if $j \neq l$ due to differences in velocities, the exponential in $M_{ijl}(t)$ is oscillatory in $\xi$. We exploit these oscillations by integrating by parts in frequency.

Thus, let $i,j,l \in \{1,\ldots,n\}$ with $j \neq l$ and let $t \in [2,T)$. Integration by parts yields
\begin{align}
\int_\R \re^{(c_l - c_j)\ri \xi s} w_j(k-\xi,s) \, \ri \xi  w_l(\xi,s) \de \xi = -\int_\R \frac{\re^{(c_l - c_j)\ri \xi s}}{(c_l-c_j) \ri s} \partial_\xi \left( w_j(k-\xi,s) \, \ri \xi w_l(\xi,s)\right) \de \xi, \label{intspace2}
\end{align}
for $s \in (0,t]$ and $k \in \R$, where we use that $w_j(\cdot,s)$ and $|\cdot| w_l(\cdot,s)$ are $L^1$-localized as $w(s) \in W^{1,1}_1(\R,\C^n)$. We employ~\eqref{linftyemb} and~\eqref{intspace2} to bound
\begin{align*}
\left\||\cdot| M_{ijl}(t)\right\|_2 &\leq C \left(\int_0^1 \left(\int_\R \left(\int_\R \left|k \re^{-d_i k^2(t-s)} w_j(k-\xi,s) \, \xi  w_l(\xi,s) \right| \de \xi\right)^2 \de k\right)^{\frac{1}{2}} \de s\right.\\
&\qquad \qquad \left. \, + \int_1^t \left(\int_\R \left(\int_\R \left|k \re^{-d_i k^2(t-s)}s^{-1} \partial_\xi \left(w_j(k-\xi,s) \, \xi  w_l(\xi,s)\right) \right| \de \xi\right)^2 \de k\right)^{\frac{1}{2}} \de s\right)\\
&\leq C \eta(t)^2 \left(\int_0^1 \frac{1}{(t-s)^{\frac{3}{4}} (1+s)} \de s + \int_1^t \frac{\ln(2+s)}{s(t-s)^{\frac{3}{4}} \sqrt{1+s} } \de s\right) \leq C\frac{\eta(t)^2}{(1+t)^{\frac{3}{4}}},
\end{align*}
and, similarly, for $a = 0,1$ we estimate
\begin{align*}
\left\||\cdot|^a M_{ijl}(t)\right\|_1 &\leq C \left(\int_0^1 \int_\R \int_\R \left|k^a \re^{-d_i k^2(t-s)} w_j(k-\xi,s) \, \xi  w_l(\xi,s) \right| \de \xi \de k \de s\right.\\
&\qquad \qquad \left. \, + \int_1^t \int_\R \int_\R \left|k^a \re^{-d_i k^2(t-s)}s^{-1} \partial_\xi \left(w_j(k-\xi,s) \, \xi  w_l(\xi,s)\right) \right| \de \xi \de k \de s\right)\\
&\leq C \eta(t)^2 \left(\int_0^1 \frac{1}{(t-s)^{\frac{1+a}{2}} (1+s)} \de s + \int_1^{\frac{t}{2}} \frac{\ln(2+s)}{s(t-s)^{\frac{1+a}{2}} \sqrt{1+s} } \de s
+ \int_{\frac{t}{2}}^t \frac{\ln(2+s)}{s(t-s)^{\frac{a}{2}} (1+s)} \de s\right)\\
&\leq C\frac{\eta(t)^2}{(1+t)^{\frac{1+a}{2}}}.
\end{align*}
The $k$-derivative of $M_{ijl}(t)$ is the sum of the following three integrals
\begin{align*}
I_{1,ijl}(k,t) &:= -\frac{\mu_{ijl}d_i}{\pi} \int_0^t k(t-s) \re^{-d_i k^2(t-s) + (c_j - c_i)\ri k s} \int_\R \re^{(c_l - c_j)\ri \xi s} w_j(k-\xi,s) \, \ri \xi  w_l(\xi,s) \de \xi \de s,\\
I_{2,ijl}(k,t) &:= \frac{\mu_{ijl}(c_j-c_i)\ri}{2\pi} \int_0^t s \, \re^{-d_i k^2(t-s) + (c_j - c_i)\ri k s} \int_\R \re^{(c_l - c_j)\ri \xi s} w_j(k-\xi,s) \, \ri \xi  w_l(\xi,s) \de \xi \de s,\\
I_{3,ijl}(k,t) &:= \frac{\mu_{ijl}}{2\pi} \int_0^t \re^{-d_i k^2(t-s) + (c_j - c_i)\ri k s} \int_\R \re^{(c_l - c_j)\ri \xi s} \partial_k w_j(k-\xi,s) \, \ri \xi  w_l(\xi,s) \de \xi \de s,
\end{align*}
which we bound one-by-one. First, using~\eqref{linftyemb} and~\eqref{intspace2}, we arrive for $a = 0,1$ at
\begin{align*}
\left\||\cdot|^a I_{1,ijl}(t)\right\|_1 &\leq C\left(\int_0^1 \int_\R \int_\R \left|k^{a+1}(t-s)\re^{-d_i k^2(t-s)} w_j(k-\xi,s) \, \xi w_l(\xi,s)\right| \de \xi \de k \de s\right.\\
&\qquad \left. + \, \int_1^t \int_\R \int_\R \left|k^{a+1}(t-s)\re^{-d_i k^2(t-s)} s^{-1} \partial_\xi \left(w_j(k-\xi,s) \, \xi w_l(\xi,s)\right)\right| \de \xi \de k \de s\right)\\
&\leq C\eta(t)^2 \left(\int_0^1 \frac{1}{(t-s)^{\frac{a}{2}}(1+s)} \de s + \int_1^t \frac{\ln(2+s)}{(t-s)^{\frac{a}{2}}s \sqrt{1+s}} \de s\right) \leq C\frac{\eta(t)^2}{(1+t)^{\frac{a}{2}}}.
\end{align*}
Second, to bound $I_{2,ijl}(t)$, we rewrite the $\xi$-derivative in~\eqref{intspace2} as
\begin{align}
\begin{split}
\partial_\xi \left( w_j(k-\xi)\, \xi w_l(\xi)\right) &= \xi w_l(\xi) \partial_\xi \left(w_j(k-\xi)\right) - (k-\xi)  w_j(k-\xi) \partial_\xi w_l(\xi) + k w_j(k-\xi) \partial_\xi w_l(\xi)\\
&\qquad \qquad  + w_j(k-\xi) w_l(\xi),\end{split} \label{yderiv}
\end{align}
for $\xi, k \in \R$, where we suppress dependency on $s \in [0,t]$. Thus,~\eqref{l2emb},~\eqref{intspace2} and~\eqref{yderiv} lead for $a = 0,1$ to the estimate
\begin{align}
\begin{split}
&\left\||\cdot|^a I_{2,ijl}(t)\right\|_1 \leq C\left(\int_{t-1}^t \int_\R \int_\R \left|k^a \re^{-d_i k^2(t-s)} s \, w_j(k-\xi,s) \, \xi w_l(\xi,s)\right| \de \xi \de k \de s\right.\\
& \quad \ \left. + \, \int_0^{t-1} \int_\R \int_\R \left|k^a \re^{-d_i k^2(t-s)} (k-\xi) w_j(k-\xi,s) \partial_\xi w_l(\xi,s)\right| \de \xi \de k \de s\right.\\
&\quad \ \left. + \, \int_0^{t-1} \int_\R \int_\R \left|k^{a+1} \re^{-d_i k^2(t-s)} w_j(k-\xi,s) \partial_\xi w_l(\xi,s)\right| \de \xi \de k \de s
+ \int_0^{t-1} \int_\R \left|k^a \re^{-d_i k^2(t-s)} w_j^{*2}(k,s)\right| \de k \de s\right)\\
&\ \ \leq C\eta(t)^2 \left(\int_{t-1}^t \frac{1}{(t-s)^{\frac{a}{2}}\sqrt{1+s}} \de s
+ \int_0^{t-1} \frac{1}{(t-s)^{\frac{1+2a}{4}} (1+s)^{\frac{3}{4}}} \de s + \int_0^{t-1} \frac{1}{(t-s)^{\frac{1+a}{2}} \sqrt{1+s}} \de s\right) \\
&\ \ \leq C\eta(t)^2 \frac{\left(\ln(2+t)\right)^a}{(1+t)^{\frac{a}{2}}}.
\end{split} \label{log2}
\end{align}
Third, we establish
\begin{align*}
\left\||\cdot|^a I_{3,ijl}(t)\right\|_1 &\leq C\int_0^t \int_\R \int_\R \left|k^a \re^{-d_i k^2(t-s)} \partial_k w_j(k-\xi,s) \, \xi w_l(\xi,s)\right| \de \xi \de k \de s\\
&\leq \int_0^t \frac{1}{(t-s)^{\frac{1+2a}{4}} (1+s)^{\frac{3}{4}}} \de s \leq C\frac{\eta(t)^2}{(1+t)^{\frac{a}{2}}},
\end{align*}
for $a = 0,1$. Hence, the bounds on $I_{b,ijl}(t)$ for $b = 1,2,3$ yield
\begin{align*}
\left\||\cdot|^a \partial_k M_{ijl}(t)\right\|_1 \leq C\frac{\eta(t)^2\left(\ln(2+t)\right)^a}{(1+t)^{\frac{a}{2}}},
\end{align*}
for $a = 0,1$. Finally, combining the estimates on $M_{ijl}(t)$ with the short-time bounds~\eqref{shorttime} and~\eqref{shorttime2}, we arrive at
\begin{align}
\left\||\cdot|^a \partial_k^b M_{ijl}(t)\right\|_1 \leq C\frac{\eta(t)^2\left(\ln(2+t)\right)^{ab}}{(1+t)^{\frac{1+a-b}{2}}}, \qquad \left\||\cdot| M_{ijl}(t)\right\|_2 \leq C\frac{\eta(t)^2}{(1+t)^{\frac{3}{4}}}, \label{nonlinestM2}
\end{align}
for $t \in [0,T)$, $a = 0,1$, $b = 0,1$ and $i,j,l \in \{1,\ldots,n\}$ with $j \neq l$.

\begin{remark} \label{logrem}
{\upshape
We note that the `artificial' $\ln(2+t)$-factor in the bound on $\||\cdot| \partial_k w(t)\|_1$ arises in the estimates~\eqref{log1} and~\eqref{log2}. We believe that such a bound can be avoided by integrating by parts in time in $\J_{2,ij}(k,t)$ and $I_{2,ijl}(t)$ in case $i \neq j$, which does not introduce singularities, since it holds $c_i \neq c_j$ and $\J_{2,ij}(k,t)$ and (the critical part in) $I_{2,ijl}(t)$ vanish at $k = 0$. However, in order not to overcomplicate the analysis we refrain from doing so.
}\end{remark}

\subsubsection{Estimates on marginal mixed-terms without derivatives}

All marginal mixed-terms without derivatives yield integrals of the form $N_{ijlm}(t)$ with $l \neq m$ in the Duhamel formulation~\eqref{Duhamel}. Since it holds $c_l \neq c_m$ if $l \neq m$, the exponential occurring in $N_{ijlm}(t)$ is oscillatory in $\zeta$. We exploit these oscillations by integrating by parts in frequency. Therefore, the procedure in this section quite similar as in~\S\ref{bound:M}.

Thus, let $i,j,l,m \in \{1,\ldots,n\}$ with $m \neq l$ and let $t \in [0,T)$ be such that $\eta(t) \leq r_0$. Integration by parts yields
\begin{align}
\int_\R \re^{(c_m - c_l)\ri \zeta s} w_l(\xi-\zeta,s)w_m(\zeta,s) \de \zeta = -\int_{\R}  \frac{\re^{(c_m - c_l)\ri \zeta s}}{(c_m - c_l) \ri s}  \partial_\zeta \left(w_l(\xi-\zeta,s)w_m(\zeta,s)\right) \de \zeta, \label{intspace}
\end{align}
for $s \in (0,t]$ and $\xi \in \R$, where we use that $w_l(\cdot,s)$ and $w_m(\cdot,s)$ are $L^1$-localized. In case $t \geq 2$, we use~\eqref{linftyemb} and~\eqref{intspace} to estimate
\begin{align*}
\left\||\cdot| N_{ijlm}(t)\right\|_2 &\leq C\left(\int_0^1 \left(\int_\R \left(\int_\R \int_\R \left|k \re^{-d_i k^2(t-s)} w_j(k-\xi,s)w_l(\xi-\zeta,s)w_m(\zeta,s)\right| \de \zeta \de \xi\right)^2 \de k\right)^{\frac{1}{2}} \de s\right.\\
&\quad \ \left.+ \, \int_1^t \left(\int_\R \left(\int_\R \int_\R \left|k \re^{-d_i k^2(t-s)} s^{-1} w_j(k-\xi,s)\partial_\zeta \left(w_l(\xi-\zeta,s)w_m(\zeta,s)\right)\right| \de \zeta \de \xi \right)^2 \de k\right)^{\frac{1}{2}} \de s\right)\\
&\leq C \eta(t)^2 \left(\int_0^1 \frac{1}{(t-s)^{\frac{3}{4}} (1+s)} \de s + \int_1^t \frac{1}{s(t-s)^{\frac{3}{4}} \sqrt{1+s} } \de s\right) \leq C\frac{\eta(t)^2}{(1+t)^{\frac{3}{4}}},
\end{align*}
and, similarly, for $a = 0,1$ we estimate
\begin{align*}
&\left\||\cdot|^a N_{ijlm}(t)\right\|_1 \leq C\left(\int_0^1 \int_\R \int_\R \int_\R \left|k^a \re^{-d_i k^2(t-s)} w_j(k-\xi,s)w_l(\xi-\zeta,s)w_m(\zeta,s)\right| \de \zeta \de \xi \de k \de s\right.\\
&\qquad \qquad \qquad \qquad \quad \left.+ \, \int_1^t \int_\R \int_\R \int_\R \left|k^a \re^{-d_i k^2(t-s)} s^{-1} w_j(k-\xi,s)\partial_\zeta \left(w_l(\xi-\zeta,s)w_m(\zeta,s)\right)\right| \de \zeta \de \xi \de k \de s\right)\\
&\qquad \leq C \eta(t)^2 \left(\int_0^1 \frac{1}{(t-s)^{\frac{1+a}{2}} (1+s)} \de s + \int_1^{\frac{t}{2}} \frac{1}{s(t-s)^{\frac{1+a}{2}} \sqrt{1+s} } \de s
+ \int_{\frac{t}{2}}^t \frac{1}{s(t-s)^{\frac{a}{2}} (1+s)} \de s\right)\\
&\qquad \leq C\frac{\eta(t)^2}{(1+t)^{\frac{1+a}{2}}}.
\end{align*}
Moreover, we establish via~\eqref{l2emb} and~\eqref{intspace}
\begin{align*}
&\left\||\cdot|^a \partial_k N_{ijlm}(t)\right\|_1 \\
&\ \ \leq C\left(\int_0^1 \int_\R \int_\R \int_\R \left|k^{a+1}d_i(t-s)\re^{-d_i k^2(t-s)} w_j(k-\xi,s) w_l(\xi-\zeta,s)w_m(\zeta,s)\right| \de \zeta \de \xi \de k \de s\right.\\
&\qquad\qquad \quad \left. + \, \int_1^t \int_\R \int_\R \int_\R \left|k^{a+1}d_i(t-s)\re^{-d_i k^2(t-s)} s^{-1} w_j(k-\xi,s) \partial_\zeta\left(w_l(\xi-\zeta,s)w_m(\zeta,s)\right)\right| \de \zeta \de \xi \de k \de s\right.\\
&\qquad\qquad \quad \left. + \, \int_0^t \int_\R \int_\R \int_\R \left|k^a(c_j-c_i)\re^{-d_i k^2(t-s)} w_j(k-\xi,s) \partial_\zeta\left(w_l(\xi-\zeta,s)w_m(\zeta,s)\right)\right| \de \zeta \de \xi \de k \de s\right.\\
&\qquad\qquad \quad \left. + \, \int_0^t \int_\R \int_\R \int_\R \left|k^a \re^{-d_i k^2(t-s)} \partial_k w_j(k-\xi,s) w_l(\xi-\zeta,s)w_m(\zeta,s)\right| \de \zeta \de \xi \de k \de s\right)\\
&\ \ \leq C\eta(t)^2 \left(\int_0^1 \frac{1}{(t-s)^{\frac{a}{2}}(1+s)} \de s + \int_1^t \frac{1}{(t-s)^{\frac{a}{2}}s \sqrt{1+s}} \de s + \int_0^t \frac{1}{(t-s)^{\frac{1+2a}{4}} (1+s)^{\frac{3}{4}}} \de s\right) \leq C\frac{\eta(t)^2}{(1+t)^{\frac{a}{2}}},
\end{align*}
for $a = 0,1$. In case $t \leq 2$, we establish the short-time bounds
\begin{align*}
&\left\||\cdot| N_{ijlm}(t)\right\|_2 \leq C\int_0^t \left(\int_\R \left(\int_\R \int_\R \left|k \re^{-d_i k^2(t-s)}  w_j(k-\xi,s) w_l(\xi-\zeta,s)w_m(\zeta,s)\right| \de \zeta \de \xi \right)^2 \de k\right)^{\frac{1}{2}} \de s\\
&\quad \leq C\eta(t)^2 \int_0^t \frac{1}{(t-s)^{\frac{3}{4}}} \de s \leq C\frac{\eta(t)^2}{(1+t)^{\frac{3}{4}}},
\end{align*}
and
\begin{align*}
&\left\||\cdot|^a \partial_k^b N_{ijlm}(t)\right\|_1 \leq C\left(\int_0^t \int_\R \int_\R \int_\R \left|k^a \re^{-d_i k^2(t-s)} \partial_k^b w_j(k-\xi,s) w_l(\xi-\zeta,s)w_m(\zeta,s)\right| \de \zeta \de \xi \de k \de s\right.\\
&\qquad \quad \left. + \, \int_0^t \int_\R \int_\R \int_\R \left|k^a\left(|k|(t-s) + |c_j-c_i|s\right) \re^{-d_i k^2(t-s)} w_j(k-\xi,s)w_l(\xi-\zeta,s)w_m(\zeta,s)\right| \de \zeta \de \xi \de k \de s\right)\\
&\quad \leq C\eta(t)^2 \int_0^t \frac{1}{(t-s)^{\frac{a}{2}}} \de s \leq C\frac{\eta(t)^2}{(1+t)^{\frac{1+a-b}{2}}},
\end{align*}
for $a = 0,1$ and $b = 0,1$. All in all, the analysis in this paragraph leads to the following nonlinear estimates
\begin{align}
\left\||\cdot|^a \partial_k^b N_{ijlm}(t)\right\|_1 \leq C\frac{\eta(t)^2}{(1+t)^{\frac{1+a-b}{2}}}, \qquad \left\||\cdot| N_{ijlm}(t)\right\|_2 \leq C\frac{\eta(t)^2}{(1+t)^{\frac{3}{4}}},
\label{nonlinestN}
\end{align}
for $t \in [0,T)$, $a = 0,1$, $b = 0,1$ and $i,j,l,m\in\{1,\ldots,n\}$ with $l \neq m$.

\subsubsection{Conclusion}

Finally, by combining~\eqref{Duhamel},~\eqref{linest},~\eqref{nonlinestR},~\eqref{nonlinestM1},~\eqref{nonlinestM2} and~\eqref{nonlinestN} we establish that, provided $v_0 \in W^{1,1}_1(\R,\C^n)$ satisfies $\|v_0\|_{\smash{W^{1,1}_1}} \leq \delta$ and $t \in [0,T)$ is such that $\eta(t) \leq r_0$, the key estimate~\eqref{etaest} holds true. This concludes, as explained in~\S\ref{sec:planM}, the proof of Theorem~\ref{mainresult1}.  $\hfill \Box$

\section{Future outlook} \label{sec:discussion}

This paper provides an alternative method to capture the effect of different velocities on the long-time dynamics of small, localized initial data in multi-component reaction-diffusion-advection systems. In combination with the earlier results in~\cite{RSDV}, we can affirm that, if each component propagates with a different velocity, then large classes of relevant and marginal nonlinearities in~\eqref{GRD} do not affect the decay of small, localized initial data. On the other hand, it is shown in~\cite[Theorem 1.4]{RSDV} that, even if each component exhibits different velocities, there are still nonlinearities which could lead to finite time blow-up of solutions with small initial data.

All in all, we are still far from a complete characterization. Perhaps the most pressing question is whether it is possible, as in two-component RDA systems, to include quadratic mixed-terms in the analysis for general multi-component RDA systems. It was already mentioned in~\cite[Section 8]{RSDV} that the method of pointwise estimates can be employed to handle quadratic mixed-terms in  $n$-component RDA systems for $n \geq 2$, if the nonlinearity has the special form $f(u,\partial_x u) = \mathrm{diag}(u_1,\ldots,u_n)g(u,\partial_x u)$ with $g \colon \R^n \times \R^n \to \R^n$ smooth, so that each term in the $i$-component has a contribution from the $i$-th component. However, it is still open how to handle quadratic mixed-terms in general $n$-component RDA systems for $n > 2$.

As outlined in Remark~\ref{quadraticmix}, one could try to extend the method in this paper to work for quadratic mixed-terms by taking stronger-than-polynomially localized initial data and by simultaneously controlling all frequency derivatives of the solution in Fourier space in the nonlinear iteration. A second, more refined, idea is to decompose the solution in Fourier space into a principal part, which is analytic in frequency and exhibits slow temporal decay, and a remainder, which decays faster in time; thus being in accordance with the algebraic-exponential decomposition of the pointwise bound~\eqref{pointwise} obtained from~\cite{RSDV}. The contributions in the Duhamel formulation coming from the principal part of a quadratic mixed-term can then be integrated by parts in frequency repeatedly to reveal additional temporal decay that arises due to differences in velocities.

Besides those future directions already discussed in~\cite[Section 8]{RSDV}, it would be interesting to extend the current method to larger classes of systems. A first gentle step would be to stay in the parabolic framework and to allow for cross-advection and cross-diffusion in~\eqref{GRD}. We expect that, after diagonalization, the current analysis or the one in~\cite{RSDV} can be employed. Another option would be to allow for spatially varying coefficients in~\eqref{GRD}. In case the spectrum of the linearization about the rest state $u = 0$ in~\eqref{GRD} is marginally stable and has multiple critical modes, differences in group velocities can, possibly after applying mode filters, be exploited. Moreover, it would be interesting to extend the current analysis beyond the parabolic framework. A natural first step in this direction would be to look at hyperbolic-parabolic systems. Finally, instead of the effect on the long-term dynamics of differences in advection between components, one could also investigate the effect of differences in nonlinear transport, for instance induced by nonlinear Burgers'-type terms.

\bibliographystyle{plain}
\bibliography{mybib}

\end{document}